\documentclass[9pt,twoside,english,reqno,a4paper]{amsart}
\usepackage{listings,graphicx,amsmath,varioref,amscd,amssymb,color,bm,stmaryrd,amsthm,amsfonts,graphics,geometry,latexsym,pgf,pst-all,soul} 
\theoremstyle{plain}
\usepackage{esint}
\usepackage{amsthm}

\theoremstyle{plain}
\newtheorem{theorem}{Theorem}[section]

\newtheorem{lemma}[theorem]{Lemma}

\newtheorem{corollary}[theorem]{Corollary}

\usepackage{geometry}
\theoremstyle{definition}
\newtheorem{defin}[theorem]{Definition}

\newtheorem{remark}[theorem]{Remark}
\newtheorem{example}[theorem]{Example}

\theoremstyle{remark}

\usepackage{fouriernc}
\usepackage[T1]{fontenc}

\def\bk{\color{black}}

\numberwithin{equation}{section}

\def\dis{\displaystyle}

\DeclareMathOperator{\R}{\mathbb{R}}

\newcommand{\car}[1]{\raise1pt\hbox{$\chi$}_{#1}}

\newcommand{\DM }{\mathcal{DM}^\infty }

\def\re{\mathbb{R}}

\usepackage{paralist}
\usepackage{epstopdf}
\usepackage[colorlinks=true]{hyperref}

\newcommand{\res}{\!\!\mathop{\hbox{
			\vrule height 7pt width .5pt depth 0pt
			\vrule height .5pt width 6pt depth 0pt}}
	\nolimits}
\begin{document}
	\title[The elliptic nonlinear transparent media equation]{Existence and regularity of solutions for the elliptic \\  nonlinear transparent media equation}

	\author[F. Balducci]{Francesco Balducci}
	\address{Francesco Balducci
		\hfill \break\indent
		Dipartimento di Scienze di Base e Applicate per l' Ingegneria, Sapienza Universit\`a di Roma
		\hfill \break\indent
		Via Scarpa 16, 00161 Roma, Italy}
	\email{\tt francesco.balducci@uniroma1.it}
	\author[F. Oliva]{Francescantonio Oliva}
	\address{Francescantonio Oliva
		\hfill \break\indent
		Dipartimento di Scienze di Base e Applicate per l' Ingegneria, Sapienza Universit\`a di Roma
		\hfill \break\indent
		Via Scarpa 16, 00161 Roma, Italy}
	\email{\tt francescantonio.oliva@uniroma1.it}
	\author[F. Petitta]{Francesco Petitta}
	\address{Francesco Petitta
		\hfill \break\indent
		Dipartimento di Scienze di Base e Applicate per l' Ingegneria, Sapienza Universit\`a di Roma
		\hfill \break\indent
		Via Scarpa 16, 00161 Roma, Italy}
	\email{\tt francesco.petitta@uniroma1.it}
	\author[M. F. Stapenhorst]{Matheus F. Stapenhorst}
	\address{Matheus F. Stapenhorst
		\hfill \break\indent
		Departamento de Matem\'atica e Computação
		Universidade Estadual Paulista,
		\hfill \break\indent
		Rua Roberto Símonsen 305, 19060-900 Presidente Prudente, SP, Brazil}
	\email{\tt m.stapenhorst@unesp.br}
	\makeatletter
	\@namedef{subjclassname@2020}{%
		\textup{2020} Mathematics Subject Classification}
	\makeatother
	

	\begin{abstract}
		In this paper we study existence and regularity of solutions to Dirichlet problems as 
		$$
		\begin{cases}
			\dis -\operatorname{div}\left(|u|^m\frac{D u}{|D u|}\right) = f & \text{in}\;\Omega,\\
			u=0 & \text{on}\;\partial\Omega,
		\end{cases}
		$$
		where $\Omega$  is an open bounded subset of $\R^N$ ($N\ge 2$) with Lipschitz boundary, $m>0$, and $f$ belongs to the Lorentz space $L^{N,\infty}(\Omega)$. In particular,  we explore the regularizing effect given by the degenerate coefficient $|u|^m$ in order to get  non-trivial and bounded solutions with no smallness assumptions on the size of the data.   \end{abstract}
	
	\maketitle
	\tableofcontents
	\section{Introduction}

	Consider the Dirichlet problem 
	\begin{equation}\label{prob}
		\begin{cases}
			\dis - \operatorname{div}\left(|u|^m\frac{Du}{|Du|}\right)=f & \text{in $\Omega$,}\\
			u=0 & \text{on $\partial \Omega$},
		\end{cases}
	\end{equation}
	where $\Omega\subset\mathbb{R}^N$ ($N\geq 2$) is a bounded domain with Lipschitz boundary, $m>0$,  and $f\in L^{N,\infty}(\Omega)$. 
	
	\medskip
	
	The class of problems as  in \eqref{prob} when $m=0$ (i.e. the case of the $1$-Laplacian)
	\begin{equation}\label{1La}
		\begin{cases}
			\dis - \operatorname{div}\left(\frac{Du}{|Du|}\right)=f & \text{in $\Omega$,}\\
			u=0 & \text{on $\partial \Omega$},
		\end{cases}
	\end{equation}
	has been widely studied in a series of papers \cite{K, CT,Demengel2005, MST1, MST2009} as an outcome  study of the asymptotic behaviour of the solutions to the problem
	\begin{equation}
		\label{PLa} 
		\begin{cases}
			-\Delta_p u =f&\quad \mbox{ in } \Omega,\\
			u=  0& \quad \mbox{ on }\partial\Omega,
		\end{cases}
	\end{equation} 
	as $p\to 1^+$ whenever the norm of $f$ is small. In \cite{K}, for instance,   the author studied the existence of solutions for case $f=1$ provided $\Omega$ has a  suitably small size. In particular, it is shown that the variational problem associated to  \eqref{PLa}    could admit  a non-trivial minimizer. 
	Most notably, it is known that solutions $u_{p} $ of the problem \eqref{PLa}
	converge to zero as $p\to1^+$ if the norm $\|f\|_{L^{N,\infty}(\Omega)}$  is small  while they blow up if the same norm is sufficiently large, see for example \cite{CT}.

	Furthermore, in \cite{MST1}, the authors demonstrated that the limit $u$ (whether trivial or non-trivial) of the solutions $u_p$ to \eqref{PLa} is a solution to \eqref{1La}, as previously defined in \cite{acm2001, AndreuMazonMollCaselles2004, AndreuCasellesMazon2004}. They also highlighted that the smallness of the norm of the datum plays a critical role in ensuring, in certain special cases, the existence of a non-trivial solution.

	Moreover,  the $L^1$ data case, even in presence of lower order terms, has also been dealt with (see for instance \cite{MST2009,  lops}). \bk Observe that solutions of problems as in \eqref{1La} are naturally set in the space of functions of bounded variation and they could admit a non-trivial jump part (see for instance \cite{K, MST1}).

	\medskip 
	Problems as in \eqref{prob} but in presence of an absorption zero order term and nonnegative data as
	\begin{equation}\label{gmp}
		\begin{cases}
			u - \dis \operatorname{div}\left(u^m\frac{Du}{|Du|}\right)=f & \text{in $\Omega$,}\\
			u=0 & \text{on $\partial \Omega$},
		\end{cases}
	\end{equation}
	have been  considered    in \cite{GMP} for   $0\leq f\in L^{\infty}(\Omega)$ (with any size),  motivated by the study of the resolvent equation of the associated evolution problem.  The fact that no restrictions are needed on the size of the datum to deduce existence of solutions is an easy consequence  of  the absorption character of the lower order term. In fact, even for $m=0$ (see for instance \cite{dass,O, OPS}) the existence of  a bounded solution of problem \eqref{gmp} can be proven without  any  smallness hypothesis on $f\in L^N(\Omega)$.

	This result is sharp for slightly more general data $f \in L^{N,\infty}(\Omega)$, although in this case, solutions of \eqref{gmp} are not expected to be bounded unless certain smallness assumptions on $\|f\|_{L^{N,\infty}(\Omega)}$ are imposed.

	\medskip 	
	Problems involving the $1$-Laplace operator as in \eqref{1La}   naturally appear in a variety of physical applications as for instance in image processing, 	but also   in the study of torsional creep of  a cylindrical bar of constant cross section in $\re^2$;  for an   account on these and further applications one may refer  to \cite{ROF, OsSe, K, ka},  but also to \cite{Sapiro, M, BCRS}, and to   the monograph \cite{AndreuCasellesMazon2004}. 	
	
	\medskip 
	
	Problems   as in \eqref{prob} (or \eqref{gmp} and it associated evolution equation)  enter in the study of the so-called nonlinear heat equation in transparent media ($m\geq1$) and its operator can be shown to be the formal limit of the porous medium relativistic operator 
	\begin{equation}\label{pmrhe}
		\varrho{\rm div } \left(\frac{{|u|^m}\nabla u}{\sqrt{u^2+\varrho^2|\nabla u|^2}}\right), \quad m>1\,,
	\end{equation}
	as the kinematic viscosity $\varrho$ tends to $\infty$ (see \cite{ACMMevo, GMP} and references therein). Equation  \eqref{pmrhe} was introduced in \cite{ros} in order to study heat diffusion in neutral gases for $m=\frac{3}{2}$.

	\medskip	
	A  further motivation comes from \cite{AV}  where the authors  pointed out  that the differential operator in \eqref{prob} is the formal limit as $p\to 1^+$ of the one of the porous medium equation in the pseudo-linear  regime, i.e., say $u\geq 0$ for simplicity, one has  
	$$
	\Delta_p u^{\frac{1}{p-1}} \to \operatorname{div}\left(u\frac{Du}{|Du|}\right),\ \ \text{as}\ \ p\to1^{+}, 
	$$
	and more in general
	$$
	\Delta_p u^{\frac{m}{p-1}} \to \operatorname{div}\left(u^m\frac{Du}{|Du|}\right),\ \ \text{as}\ \ p\to1^{+}, 
	$$
	for any $m>0$.
	
	\medskip 
	
	A further motivation in the study of problems as \eqref{prob} comes from  the connection with some 1-Laplace type   problems having natural growth in   the lower order term of order one.  
	Indeed, for a nonnegative smooth $u$ one formally has
	\begin{equation}\label{equ comp}
		\begin{aligned}
			-\operatorname{div}\left(u^{m}\frac{D u}{|D u|}\right)&=-\sum_{i}\frac{\partial}{\partial x_i}\left(\frac{u^{m}u_{x_i}}{|D u|}\right)\\
			&=-\frac{m}{|D u|}\sum_{i}u^{m-1}u_{x_i}^2-u^{m}\sum_{i}\frac{\partial}{\partial x_{i}}\left(\frac{u_{x_i}}{|D u|}\right)\\
			&=-mu^{m-1}|D u|-u^{m}\Delta_1 u
		\end{aligned}
	\end{equation}
	which gives that 
	$$
	-\operatorname{div}\left(u^{m}\frac{D u}{|D u|}\right)=f,$${ is formally equivalent to }$$-\Delta_{1}u= m \frac{|D u|}{u}+\frac{f}{u^m}. 
	$$

	Hence, the equation we focus on in this paper  formally  represents a borderline, and new,  case  of the  singular elliptic equations involving the 1-Laplace operator and  natural growth   gradient terms studied in  \cite{GOP} where the authors considered boundary value problems governed by 
	$$
	-\Delta_{1}u= m\frac{|D u|}{u^\theta}+\frac{f}{u^m}, 
	$$ \bk 
	with $0<\theta<1$ and $m> 0$. 
	
	\medskip 
	
	The goal of our work consists in studying the  existence and the qualitative properties of  solutions to problems as in \eqref{prob} under minimal assumptions on the data. Of particular interest with respect to the case of the $1$-Laplacian is the investigation of the regularizing effects provided by the nonlinear coefficient $|u|^m$.  
	
	\smallskip
	
	We shall see that a  bounded non-trivial  solution to \eqref{prob} does exist for any $f\in L^{N,\infty}(\Omega)$ no matter  the size of $\|f\|_{L^{N,\infty}(\Omega)}$ (to be compared with the mentioned results of \cite{K, CT, MST1}). 
	
	\smallskip
	
	As already mentioned, in many known results the regularizing effect on the existence and smoothness of solutions for problems as in \eqref{prob} are driven by the presence of suitable lower order absorption terms (see   \cite{GMP,DGOP,lops,O, OPS} and references therein).  In the present work instead we exploit the sole effect of the coefficient $|u|^m$ in order to avoid the request of any smallness assumptions on the data. Moreover, using an idea of \cite{GMP} we are able to prove that solutions  of \eqref{prob} (or suitable truncations of them) are globally in  $BV(\Omega)$ (here $BV(\Omega)$ denotes the space of functions of bounded variation on $\Omega$) but  without  any  jump part  (we will provide  more precise definitions later); this is a typical  feature that solutions of problems as in \eqref{1La}  usually enjoy  in presence of first order terms with natural growth (see for instance \cite{ads, dass, GOP, BOP}). 
	
	\smallskip
	
	We stress that, through a fine truncation argument,   a  convenient definition of solutions to  problem 
	\eqref{prob} can be given without any sign assumption on the datum $f\in L^{N,\infty}(\Omega)$ (to be compared with the results in \cite{GMP}). We stress that this notion of solution is  inspired by the entropy notion  introduced in \cite{acmNA} (see also \cite{acmJEMS}) in order to deal with general  flux-saturated operators in $\re^N$ and in \cite{acmMA} for the associated Dirichlet problem (see also the survey \cite{calvo} for further details on the subject). 
	With this definition at hand we shall  prove existence of a non-trivial solution without jump part under minimal hypothesis on the data.

	\smallskip

	In this case it is noteworthy to deal with the appropriate definition of the sets $\{u>0\}\cap\partial\Omega$
	and $\{u<0\}\cap\partial\Omega$ even for functions that, a priori, do not admit any trace in the classical sense (see Definition \ref{defin set bd} below) in order to give sense to the homogeneous boundary datum. 
	
	\medskip 
	
	Our general strategy, in order to show that problem \eqref{prob} has a non-trivial solution,  could be briefly summarized as follows:  first, we obtain a sequence $u_{\varepsilon}$ in $H_{0}^{1}(\Omega)$ of solutions of the auxiliary problem
	\begin{equation} \label{e.problemapprox2}
		\begin{cases}
			\dis -\operatorname{div}\left(|u_{\varepsilon}|^{m}\frac{\nabla u_{\varepsilon}}{|\nabla u_{\varepsilon}|_\varepsilon}+\varepsilon \nabla u_{\varepsilon}\right)= f &\text{in $\Omega$,} \\
			u_{\varepsilon} =  0 &\text{ on 
				$\partial \Omega$},
		\end{cases}
	\end{equation}
	where
	$$
	|\xi|_{\varepsilon}=\sqrt{|\xi|^2+\varepsilon^2}\text{ for all }\xi\in \mathbb{R}^{N}.
	$$
	Next we show that the sequence $u_{\varepsilon}$ is uniformly bounded in $L^{\infty}(\Omega)$ and that the sequence $|u_{\varepsilon}|^{m+1} $ is uniformly bounded in $BV(\Omega)$ with respect to $\varepsilon$. 	
	As a consequence of the fact that $BV(\Omega)$ is compactly embedded in $L^{r}(\Omega)$ for all $1\leq r<\frac{N}{N-1}$, we can detect  a function $u$ such that, up to a subsequence,
	$$
	u_{\varepsilon}\to u\text{ in }L^{r}(\Omega)\text{ for all }1\leq r<\frac{N}{N-1}(m+1).
	$$
	Finally we pass to the limit  in \eqref{e.problemapprox2}   by showing that $u$ is a solution of problem \eqref{prob} in a suitable  sense (see Definition \ref{def sol changing}). As we said, our results  should  be compared with other ones in the literature.  In this work, it is worth mentioning that, even without any assumption on the size of $\|f\|_{L^{N,\infty}(\Omega)}$, we still reach a non-trivial and finite   solution. \bk  
	
	\medskip
	
	Let us outline the structure of the paper. 
	In Section \ref{s.prelim}, we give some necessary preliminary definitions and we briefly  summarize the Anzellotti-Chen-Frid  pairing theory. 
	In Section \ref{s.existence} we introduce our notion of solution to problem \eqref{prob} and, for the sake of exposition,  we state our main result in case of a  nonnegative datum $f$.
	Section \ref{s.approxsolutions} is devoted to basic a priori estimates on the sequence  $u_{\varepsilon}$ of problem \eqref{e.problemapprox2} that, furthermore,  are shown to be  uniformly bounded in $L^{\infty}(\Omega)$. We also obtain an estimate in $BV(\Omega)$ for a suitable power of the solution and we detect a limit function $u$.
	In Section \ref{s.proof}, we pass to the limit in \eqref{e.problemapprox2} and we show that problem \eqref{prob} has a solution for a nonnegative $f$.
	Section \ref{s.signchanging} contains  a (technical) extension of the previous  results to the case of a possible  sign changing datum $f$ in $L^{N,\infty}(\Omega)$.  Finally,  in order to illustrate our results,  in Section \ref{sec7} we construct some explicit examples   including the solution to the torsion problem associated to the transparent media equation.  
	
	\section{Preliminary tools}\label{s.prelim}
	\subsection{Basic notation}
	Let $\Omega \subset \mathbb{R}^N$ ($N \ge 2$) be a bounded open set with Lipschitz continuous boundary. We denote  $\mathcal{H}^{N-1}(\Gamma)$ \bk as the  measure of  an  $(N-1)$-dimensional set    $\Gamma$, while $|E|$ indicates the $N$-dimensional Lebesgue measure of $E$.
	
	The space $\mathcal{M}(\Omega)$ is the set of Radon measures with finite total variation on $\Omega$.
	\\Let us introduce the following truncation function, for $-\infty \leq a < b \leq \infty$:
	$$T_a^b(s):=\max(\min(b,s),a).$$
	In the particular case  $a= -b$ with $b>0$ we will use the standard truncating function:
	\begin{equation}\label{troncata}
		T_b(s):=\max(\min(b,s),-b).
	\end{equation}
	Moreover, we will also use for every $k>0$:
	\begin{equation}\label{tronc 2}
		G_k(s):=s-T_k(s).
	\end{equation} 
	For the sake of simplicity, and to avoid ambiguity, we will often use the following agreement:
	$$\int_{\Omega} f := \int_{\Omega} f(x) \, \ensuremath{ d}x.$$
	If not otherwise specified, we denote by $C$ several positive constants whose values may change from line to line and, occasionally, within the same line. These values will depend only on the   data but never   on the indices of the sequences we introduce over time. Also, for simplicity, we will not relabel an extracted compact subsequence when no confusion is possible.
	
	\subsection{Lorentz spaces} The Lorentz space $L^{p,q}(\Omega)$ is defined as the space of measurable functions $u$ such that the quantity
	\begin{equation*}\label{deflorentz}
		\|u\|_{L^{p,q}(\Omega)}=\dis \begin{cases}
			\dis \left(\int_{0}^{\infty}[t^{1/p}u^{*}(t)]^q\frac{\,dt}{t}\right)^{\frac{1}{q}} & 1<p<\infty,\,1<q<\infty, \\
			\dis \sup_{t>0} t^{\frac{1}{p}}u^{*}(t) & 1<p<\infty,\,q=\infty,
		\end{cases}
	\end{equation*}
	is finite. Here, $u^{*}(t)$ denotes the decreasing rearrangement of $u$, given by
	$$
	u^{*}(s)=\sup\{t>0: |\{|u|>t\}|>s\}.
	$$
	We recall that $L^{p,p}(\Omega)=L^{p}(\Omega)$. Moreover, for all $1<q<p<s<\infty$, the following inclusions hold:
	$$
	L^{s}(\Omega)\subset L^{p,1}(\Omega)\subset L^{p,q}(\Omega)\subset  L^{p}(\Omega)\subset L^{p,s}(\Omega)\subset L^{p,\infty}(\Omega)\subset L^{q}(\Omega).
	$$
	A Sobolev type inequality in Lorentz spaces holds; in particular there exists a constant $\widetilde{\mathcal{S}}_1>0$ such that
	\begin{equation}\label{des lorentz}
		\|u\|_{L^{1^*,1}(\Omega)} \le \widetilde{\mathcal{S}}_1 \|\nabla u\|_{L^1(\Omega)}\text{ for every }u\in W_{0}^{1,1}(\Omega),
	\end{equation}
	where $1^*=\frac{N}{N-1}$ stands for  the standard Sobolev conjugate exponent. 
	
	\medskip 
	
	Let us also mention that, for all $1\leq p_{1},p_{2}<\infty$ and $1\leq q_{1},q_{2}\leq\infty$ with $\frac{1}{p_1}+\frac{1}{p_2}=\frac{1}{q_1}+\frac{1}{q_2}=1$, a Young inequality applies in Lorentz spaces, i.e.  
	\begin{equation}\label{holder lor}
		\|fg\|_{L^{1}(\Omega)}\leq C\|f\|_{L^{p_1,q_1}(\Omega)}\|g\|_{L^{p_2,q_2}(\Omega)}.
	\end{equation}
	For more details see \cite{AL, ON, Z}.
	\subsection{BV and TBV spaces}
	For an introduction on  $BV$  spaces we refer to \cite{AFP} from where most of our notations are taken. Briefly, we recall that the set $BV(\Omega)$ consists   of those  functions $u \in L^1(\Omega)$ whose distributional gradient is in $\mathcal{M}(\Omega)^N$, and  it  is a Banach space  endowed with the norm:  $$\|u\|_{BV(\Omega)} =\int_{\partial\Omega} |u| \, d\mathcal{H}^{N-1} + \int_{\Omega} |Du|,$$ where $\int_\Omega |Du|$ denotes the total variation of the measure $Du$ over $\Omega$, i.e. 
	\begin{equation*}
		\label{TV}
		\int_\Omega |D u|=\sup\left\{\int_{\Omega}u\,\mbox{div}\,\phi : \phi\in C^1_c(\Omega,\mathbb{R}^N),\quad\|\phi\|_{L^\infty(\Omega)^N}\leq 1\right\}.
	\end{equation*}
	We also recall that an equivalent norm for $BV(\Omega)$ is given by
	$$\||u|\|_{BV(\Omega)} =\int_{\Omega} |u|   + \int_{\Omega} |Du|.$$ 
	Recall that, for $u\in L^1(\Omega)$, $u$ has an approximate limit at $x\in \Omega$ if there exists $\widetilde{u}(x)$ such that 
	$$
	\lim_{\rho\downarrow 0} \fint_{B_{\rho}(x)}|u(y)-\widetilde{u}(x))|\,dy=0,
	$$
	where $\fint_{E}f=\frac{1}{|E|}\int_{E}f$; such points are called Lebesgue points of $u$ and the set of these points is denoted by $L_u$. The set where this property does not hold is denoted by $S_u$. This is a $\mathcal{L}^N-$negligible Borel set \cite[Proposition 3.64]{AFP} . We say that $x$ is an approximate jump point of $u$ if there exists $u^+(x)\neq u^{-}(x)$ and $ \nu\in S^{N-1}$  such that
	\begin{eqnarray*}
		\lim_{\rho \downarrow 0}\fint_{B^+_\rho(x, \nu) }|u(y)-u^+(x)|\, dy&=&0,\\
		\lim_{\rho \downarrow 0}\fint_{B^{-}_\rho(x, \nu) }|u(y)-u^{-}(x)|\, dy&=&0,
	\end{eqnarray*}
	where
	\begin{eqnarray*}
		B^+_\rho (x, \nu)&=&\{y\in B_\rho(x)\>:\> \langle y-x, \nu\rangle >0\},\\
		B^{-} _\rho (x, \nu)&=&\{y\in B_\rho(x)\>:\> \langle y-x, \nu\rangle <0\}.
	\end{eqnarray*}
	The set of approximate jump points is denoted by $J_u$. The set $J_u$ is a Borel subset of $S_u$ \cite[Proposition 3.69]{AFP} and $\mathcal{H}^{N-1}(S_u \setminus J_u) = 0,$ if $u \in BV(\Omega)$. Moreover, up to a $(N-1)$-negligible set $J_u$ is an $\mathcal{H}^{N-1}$-  rectifiable set and an orientation $\nu_u (x)$ is defined for $\mathcal{H}^{N-1}$-a.e. $x \in J_u$. 
	
	For $u\in L^1(\Omega)$, $u^*\>:\>\Omega\setminus(S_u\setminus J_u)\to \mathbb{R}$ is called the precise representative of $u$ if 
	\begin{equation*}
		\label{CanRep}
		u^*(x)=
		\left\{
		\begin{array}{lcr}
			\widetilde{u}(x)&\mbox{ if }& x\in \Omega\setminus S_u,\\
			\dis \frac{u^+(x)+u^{-}(x)}{2}& \mbox{ if }& x\in J_u.
		\end{array}
		\right.
	\end{equation*}
	Let us stress that any $u \in BV(\Omega)$ can be characterized by its precise representative $u^*$, which coincides with its Lebesgue representative on $L_u$ and that $u^*$ is well-defined $\mathcal{H}^{N-1}$-a.e. because the set $S_u \setminus J_u$ is $\mathcal{H}^{N-1}$-negligible.
	
	\medskip 
	For $u\in BV(\Omega)$, we write
	$$
	Du=D^{a}u+D^{j}u+D^{c}u,
	$$
	where $D^a, D^j$ and $D^{c}$ denote respectively the Lebesgue, the jump and the Cantor parts of $Du$. This decomposition is made up of mutually  orthogonal parts. Moreover, sets of finite $\mathcal{H}^{N-1}$ measure are read only by the jump part $D^{j} u$, i.e., $(D^{a}u)(E)=(D^{c}u)(E)=0$ for all $\mathcal{H}^{N-1}$ measurable sets $E$ with $\mathcal{H}^{N-1}(E)<\infty$. 
	
	Especially, when $D^j u =0$, it means that $\mathcal{H}^{N-1}(J_u)=0$ or, equivalently, that $Du= \widetilde{D}u$ where $\widetilde{D}u=D^a u + D^c u$.
	Let us specify we will write $u$ instead of $u^*$ when we integrate against a measure absolutely continuous with respect to $\mathcal{H}^{N-1}$ as no ambiguity is possible.
	
	\medskip 
	
	 Let us also denote by 
	$$DBV(\Omega):=\{u \in BV(\Omega) : D^j u=0\}.$$ 
	
	\medskip
	
	Now let us recall some  weak lower semicontinuity properties in $BV(\Omega)$ (see  \cite[Proposition 3.6]{AFP}).  
	If one  considers a sequence $u_n \in BV(\Omega)$ such that $u_n \to u$ strongly in $L^1(\Omega)$ with $u \in BV(\Omega)$. Then \begin{equation}\label{sci BV}
		\int_\Omega |Du| \varphi + \int_{\partial\Omega} |u| \varphi\, d\mathcal{H}^{N-1} \le \liminf_{n \to \infty} \int_\Omega |Du_n|\varphi +\int_{\partial\Omega} |u_n| \varphi \, d\mathcal{H}^{N-1}  \quad \text{for all $ 0 \le \varphi \in C^1(\overline{\Omega})$,}
	\end{equation}  
	notice in particular that \eqref{sci BV} holds with $\varphi=1$, and that 
	\begin{equation}\label{sci BV0}
		\int_\Omega |Du| \varphi   \le \liminf_{n \to \infty} \int_\Omega |Du_n|\varphi   \quad \text{for all  $0 \le\varphi \in C^1_c (\Omega)$.}
	\end{equation}  
	We also  recall the Sobolev embedding $BV(\Omega)\hookrightarrow L^{1^*}(\Omega)$, that is 
	\begin{equation}\label{e.Lstar}
		\|u\|_{L^{1^*}(\Omega)} \le \mathcal{S}_1 \|u\|_{BV(\Omega)}\text{ for every }u\in BV(\Omega),
	\end{equation}
	where  $\mathcal{S}_1>0$ is the best constant of this embedding. We further remark that the embedding $BV(\Omega)\hookrightarrow L^{r}(\Omega)$ with $1\leq r<1^*$ is compact (for more details \cite[Theorem 3.23]{AFP}). 
	
	\medskip 
	
	Let us recall the chain rule formula for functions in $BV(\Omega)$ (for instance see \cite[Theorem 3.99]{AFP}).
	\begin{theorem}\label{t chain rule}
		Let  $u \in BV(\Omega)$ and let $\Phi : \mathbb{R} \to \mathbb{R}$ be a Lipschitz function. Then $v=\Phi(u) \in BV(\Omega)$ and it holds 
		\begin{equation*}
			Dv=\Phi'(\widetilde{u})\widetilde{D}u + \left(\Phi(u^+)-\Phi(u^-)\right)\nu_u \mathcal{H}^{N-1}\res J_u.
		\end{equation*}
		Moreover if $D^j u =0$ then 
		\begin{equation}\label{chain rule senza j}
			\widetilde{D}v= \Phi'(\widetilde{u}) \widetilde{D}u.
		\end{equation}
	\end{theorem}
	For our purposes,  as in \cite{GMP}, from which we mainly derive the following notation,  let us also  introduce the space of functions whose suitable truncations belong to $BV(\Omega)$, that is 
	$$
	TBV(\Omega):=\{u\in L^{1}(\Omega):F(u^{+}), F(u^{-})\in BV(\Omega) \text{ for all $a>0$, $F \in W_{a}^{1,\infty}$} \},
	$$	
	where $u^{+}=\max\{u,0\}$ and $u^{-}=\max\{-u,0\}$ and
	$$
	W_{a}^{1,\infty}=W^{1, \infty}([0, \infty); [a, \infty)).
	$$
	In particular, $TBV(\Omega)$ may be equivalently defined as 
	$$
	TBV(\Omega):=\{u\in L^{1}(\Omega): \, T_a^b(u), T_{-b}^{-a}(u)\in BV(\Omega), \quad \text{for all $0<a<b \le \infty$}\}.
	$$ 
	For more details see \cite[Remark 4.27]{AFP}. It is known that  nonnegative functions in $TBV(\Omega)$ admit a trace, as proven in \cite[Lemma 5.1]{GMP}.
	\begin{lemma}\label{lemma trace pos}
		Let $\Omega$ be a bounded open set with Lipschitz boundary and  $u$ a nonnegative function in $ TBV(\Omega)$.  Then, there exists $u^{\Omega}\in L^{1}(\partial\Omega;[0,\infty))$ such that
		$$
		\lim_{\rho\to0}\fint_{\Omega\cap B_{\rho}(x_0)}|u(x)-u^{\Omega}(x_0)|\,dx=0\text{ for $\mathcal{H}^{N-1}$-a.e $x_0\in\partial\Omega$.}
		$$
		Moreover 
		$$
		u^{\Omega}=\lim_{a\to0^+}(T_{a}^{\infty}(u))^{\Omega}\quad\mathcal{H}^{N-1}\text{-a.e in $\partial\Omega$,}
		$$	 
		and
		$$
		F(u^{\Omega})=(F(u))^{\Omega}\text{ for all } F\in W_{a}^{1,\infty}.
		$$
	\end{lemma}
	We give a notion of trace in the boundary of $\Omega$ for functions in $TBV(\Omega)$. Recall that if $u\in TBV(\Omega)$, then $u^{+}$ and $u^{-}$ admit a trace in the sense of Lemma \ref{lemma trace pos}. 
	\begin{defin}\label{defin set bd}
		Let $u\in TBV(\Omega)$. We define
		$$
		\{u>0\}\cap\partial\Omega=\{(u^{+})^{\Omega}>0\},
		$$
		and
		$$
		\{u<0\}\cap\partial\Omega=\{(u^{-})^{\Omega}>0\}.
		$$
	\end{defin}
	With the above definition in force, one can show the following technical result.
	\begin{lemma}\label{lemma well posed bd}
		Let $u\in TBV(\Omega)$. Then
		$$
		\mathcal{H}^{N-1}\left(\{(u^{+})^{\Omega}>0\}\cap\{(u^{-})^{\Omega}>0\}\right)=0.
		$$	
	\end{lemma}
	\begin{proof} 
		We begin by observing that $(u^+)^{\Omega}$ is well defined $\mathcal{H}^{N-1}$ a.e.  in $\partial\Omega$. Let $x_0\in\partial\Omega$ be a point such that $(u^+)^{\Omega}(x_0)>0$. From Lemma \ref{lemma trace pos}, there exists $c_0,a_0>0$ such that
		\begin{equation}\label{equ hyp trace}
			(T_{a}^{\infty}(u(x_0)))^{\Omega}>c_{0}\text{ for all }0<a\leq a_0. 
		\end{equation}
		Since $u\in TBV(\Omega)$, we know that the function $v_{a}=T_{a}^{\infty}(u)+ T_{-\infty}^{-a}(u)\in BV(\Omega)$. Thus,
		\begin{equation}\label{equBVtrace}
			(v_a^+)^{\Omega}+(v_a^-)^\Omega=(v_{a}^++v_{a}^{-})^{\Omega}=(\max\{v_{a}^+,v_{a}^-\})^{\Omega}=\max\{(v_a^+)^{\Omega},(v_a^-)^{\Omega}\}\quad\mathcal{H}^{N-1}\text{-a.e }.
		\end{equation}
		Since $v_{a}^{+}=T_{a}^{\infty}(u)-a$, inequality \eqref{equ hyp trace} implies that	
		$$
		(v_a^+(x_0))^{\Omega}=(T_{a}^{\infty}(u(x_0))-a)^{\Omega}=(T_{a}^{\infty}(u(x_0)))^{\Omega}-a>c_{0}-a\text{ for all }0<a\leq a_0.
		$$
		Thus, $(v_a^+(x_0))^{\Omega}>0$ for $0<a<\min\{a_0,c_0\}$. As a consequence, \eqref{equBVtrace} yields $(v_a^-)^{\Omega}(x_0)=0$ for all $0<a<\min\{a_0,c_0\}$. Thus,
		$$
		(v_a^-)^{\Omega}(x_0)=(T_a^{\infty}(u^-)-a)^{\Omega}(x_0)=(T_a^{\infty}(u^-)(x_{0})^{\Omega}-a=0\text{ for all }0<a<\min\{a_0,c_0\}.
		$$
		We conclude that
		$$
		u^{-}(x_{0})^{\Omega}=\lim_{a\to0^+}(T_a^{\infty}(u^-)(x_{0}))^{\Omega}=0.
		$$
		This proves the result.
	\end{proof}	
	
	\medskip

	Given $u\in L^{1}_{\rm loc}(\Omega)$, the upper and the lower approximate limits of $u$ at the point $x\in\Omega$ are defined respectively as
	$$
	u^{\lor}(x):=\inf\{t\in\mathbb{R}:\lim_{\rho \downarrow 0}\rho^{-N}|\{u>t\}\cap B_{\rho}(x)|=0\},
	$$
	$$
	u^{\land}(x):=  \sup \bk\{t\in\mathbb{R}:\lim_{\rho \downarrow 0}\rho^{-N}|\{u<t\}\cap B_{\rho}(x)|=0\}.
	$$
	We let $S_{u}^{*}=\{x\in\Omega: u^{\land}(x)<u^{\lor}(x)\}$ and we  define
	$$
	DTBV^+(\Omega):=\{u\in TBV(\Omega): u\text{ is nonnegative and  }\mathcal{H}^{N-1}(S_{u}^{*})=0\},
	$$	
	and
	$$
	DTBV(\Omega):=\{u\in TBV(\Omega): u^{+},u^{-}\in DTBV^+(\Omega)\}.
	$$	
	The set of weak approximate jump points of a function $u\in L^1_{\rm{loc}}(\Omega)$ is the subset $J_{u}^*$ of $S_{u}^*$ such that there exists a unit vector $\nu^*_u(x) \in \mathbb{R}^N$ such that the weak approximate limit of the restriction of $u$ to the hyperplane $H^+:=\{y\in \Omega: (y-x)\cdot v_{u}^{*}(x)>0 \}$ is $u^\lor(x)$ and the weak approximate limit of the restriction of $u$ to the hyperplane $H^-:=\{y\in \Omega:  (y-x)\cdot v_{u}^{*}(x)<0 \}$ is $u^\land(x)$. Under the assumption that $u\in L^{1}_{\rm{loc}}(\Omega)$, it can be shown that $J_u \subseteq J_u^*$ and
	$$u^\lor(x)=\max\{u^+(x), u^-(x)\}, \quad u^\land(x)=\min\{u^+(x), u^-(x)\}, \quad \nu_u^*(x)=\pm\nu_u(x) \quad \forall x \in J_u,$$ for more details see \cite[p.237]{AFP}. 
	
	Let us state  a peculiar property of  nonnegative functions in $TBV(\Omega)$ (see \cite[Lemma 2.1]{GMP} and \cite[Theorem 4.34]{AFP}) that will be the key in order to prove that the solutions we find belong to $DTBV(\Omega)$. 
	\begin{lemma}\label{salti in TBV}
		Let $u \in TBV(\Omega) \cap L^\infty(\Omega)$ be nonnegative. Then
		\begin{enumerate}
			\item $S^*_u=\bigcup_{a>0}S_{T^\infty_a(u)}$ and $$u^\lor(x)=\lim_{a \to 0^+}(T^\infty_a(u))^\lor (x), \qquad u^\land(x)=\lim_{a \to 0^+}(T^\infty_a(u))^\land (x). $$
			\item $S^*_u$ is countably $\mathcal{H}^{N-1}$-rectifiable and $\mathcal{H}^{N-1}(S^*_u \setminus J^*_u)=0$.
		\end{enumerate}
	\end{lemma}
	We finish this summary concerning $TBV(\Omega)$ by explicitly remarking that the coarea formula (see \cite[Theorem 3.40]{AFP}) implies that the sets $\{u>a\}$ and $\{u<-a\}$ are of finite perimeter for almost every $a>0$ provided $u\in TBV(\Omega)$. Consequently, the functions $\chi_{\{a<u<b\}}$ and $\chi_{\{-b<u<-a\}}$ belong to $BV(\Omega)$ for almost all $a,b>0$.\bk
	\subsection{The Anzellotti-Chen-Frid theory}
	In this section we summarize the theory of pairings due to  Anzelotti (\cite{A}, see also \cite{CF}). First we define
	$$
	\DM(\Omega)=\{z\in L^{\infty}(\Omega)^N	:\operatorname{div}z\in \mathcal{M}(\Omega)\}.
	$$
	In \cite[Theorem 1.2]{A}, it is shown that there exists a linear operator $[\cdot,\nu]:\DM(\Omega)\to L^{\infty}(\partial\Omega)$ such that \begin{equation} \label{riferz}\|[z,\nu]\|_{L^{\infty}(\partial\Omega)}\leq\|z\|_{L^{\infty}(\Omega)^N} \ \text{for all}\ z\in\DM(\Omega),\end{equation} and
	$$
	[z,\nu](x)=z(x)\cdot \nu(x)\ \  \text{ for all} \ \  x\in\partial\Omega \ \  \text{if}\ \ z\in C^{1}(\overline{\Omega})^N.
	$$
	Moreover, in \cite[Proposition 3.1]{CF}, the authors show that $\operatorname{div}z$ is absolutely continuous with respect to $\mathcal{H}^{N-1}$ for all $z\in \DM(\Omega)$. Consequently, the functional $(z, Du)\in \mathcal{D}'(\Omega)$ given by 
	\begin{equation}\label{pairing}
		\langle (z, Du),\varphi\rangle=-\int_{\Omega}u^*\varphi \operatorname{div}z-\int_{\Omega}uz \cdot\nabla\varphi \quad \text{for all $\varphi \in C^{1}_c(\Omega)$,}
	\end{equation}
	is well defined for all $u\in BV(\Omega)\cap L^{\infty}(\Omega)$. 
	
	The distribution $(z, Du)$ is a Radon measure having  finite total variation and, for any $v\in BV(\Omega)\cap L^{\infty}(\Omega)$,  it satisfies
	\begin{equation} \label{ec:2}
		|\langle (z, Dv), \varphi\rangle| \le \|\varphi\|_{L^{\infty}(\omega) }\| z
		\|_{L^{\infty}(\omega)^N} \int_{\omega} |Dv|\,,
	\end{equation}
	for all open sets $\omega \subset\subset \Omega$ and for all $\varphi\in C_c^1(\omega)$, in particular $ |(z, Dv)|\ll |D v|$ as measures. 
	\\ Furthermore, the following result, which extends Green's identity, holds.
	\begin{lemma}\label{campo prodotto}
		Let $z \in \mathcal{DM}^\infty(\Omega)$ and $u \in BV(\Omega)\cap L^\infty(\Omega)$. Then the functional $\left(z, Du\right)\in \mathcal{D}'(\Omega)$ is a Radon measure which is absolutely continuous with respect to $|Du|$. Moreover \begin{equation}\label{int per parti}
			\int_{\Omega} u^* 	\, \operatorname{div}z +  \int_\Omega(z, Du) \bk =\int_{\partial\Omega}u[z,\nu] \, d\mathcal{H}^{N-1},
		\end{equation}
		\begin{equation}\label{= misure}
			\operatorname{div}(uz)=u^* \operatorname{div}(z)+(z,Du) \quad \text{as measures,}
		\end{equation}
		and
		\begin{equation}\label{= al bordo mis}
			[uz,\nu]=u[z,\nu] \quad \text{$\mathcal{H}^{N-1}$-a.e. on $\partial\Omega$.}
		\end{equation}
	\end{lemma}
	Under the same assumptions on $z$ and $u$ we indicate by $\theta(z, Du, x)$ the Radon-Nikod\'ym derivative of $(z,Du)$ with respect to $|Du|$, i.e. we have
	$$(z,Du)=\theta(z, Du, x)|Du| \quad \text{as measures in $\Omega$.}$$
	A chain rule for the Radon-Nykodim derivative holds (see \cite[Proposition 4.5 (iii)]{CDC}).
	\begin{lemma}
		Let $z\in\DM_{\rm{loc}}(\Omega)$, $u\in BV_{\rm{loc}}(\Omega)\cap L^{\infty}_{\rm{loc}}(\Omega)$ and $h:\mathbb{R}\to\mathbb{R}$ be a non-decreasing locally Lipschitz function. Then
		\begin{equation}\label{chain rule pairing}
			\theta(z,D h(u),x)=\theta(z,Du,x),\quad\text{for $|Dh(u)|$-a.e. $x\in\Omega$.}
		\end{equation}	  
	\end{lemma}
	\begin{remark}
		If $z\in\DM(\Omega)$ satisfies $-\operatorname{div}z=f\in L^{1}(\Omega)$, then
		\begin{equation}\label{def pair div l1}
			\langle (z, Du),\varphi\rangle=\int_{\Omega}u\varphi f-\int_{\Omega}uz \cdot\nabla\varphi\text{ for all $\varphi\in C^{1}_c(\Omega)$},
		\end{equation} 
		and \eqref{int per parti} becomes
		\begin{equation}\label{int per parti L1}
			-\int_{\Omega} uf + \int_{\Omega}(z, Du)=\int_{\partial\Omega}u[z,\nu] \, d\mathcal{H}^{N-1}.
		\end{equation}   
		Expression \eqref{int per parti L1} will be used throughout the text.
	\end{remark}
	Also useful to us is the fact that one may define the normal trace $[z,\Sigma]^{\pm}$ of a  vector field $z\in\DM(\Omega)$ on an oriented $C^1$-	hypersurface $\Sigma\subset\Omega$ by
	\begin{equation}\label{e.hypersurf}
		[z,\Sigma]^{\pm}:=[z,\nu_{\Omega^\pm}],
	\end{equation}
	where $\Omega^{\pm}\subset\subset\Omega$ are open $C^{1}$-domains such that $\Sigma\subset\partial\Omega^{\pm}$ and $\nu_{\Omega^\pm}=\pm\nu_{\Sigma}$. It can be proven that definition \eqref{e.hypersurf} does not depend on the particular choice of $\Omega^\pm$ up to a set of zero $\mathcal{H}^{N-1}$ measure. Furthermore, according to \cite[Proposition 3.4]{ACM}, it holds 
	\begin{equation}\label{e.rect1}
		\left(\operatorname{div} z\right)\res\Sigma =\left(\left[z, \Sigma\right]^+ - \left[z, \Sigma\right]^-\right) \mathcal{H}^{N-1}\res\Sigma.
	\end{equation}
	By localization, this notion can be extended to oriented countably $\mathcal{H}^{N-1}$-rectifiable sets $\Sigma$. In this way, it is possible to extend \eqref{e.rect1} to get the following result (see \cite[Lemma 2.4]{GMP}). 
	\begin{lemma}\label{lem sigma}
		Let $z \in \mathcal{DM}^\infty(\Omega)$ and let $\Sigma \subset \Omega$ be an orientated countably $\mathcal{H}^{N-1}$-rectifiable set. Then
		$$\left(\operatorname{div} z\right)\res\Sigma =\left(\left[z, \Sigma\right]^+ - \left[z, \Sigma\right]^-\right) \mathcal{H}^{N-1}\res\Sigma.$$
	\end{lemma}
	As a consequence of Lemma \ref{campo prodotto}, we get (see \cite[Lemma 2.5]{GMP}) the following. 
	\begin{lemma}\label{lem normale}
		Let $u \in BV(\Omega) \cap L^\infty(\Omega)$ and $z \in \mathcal{DM}^\infty(\Omega)$. Then
		\begin{equation}\label{saltando sul bordo}
			\left[u z, \nu_u\right]^\pm = u^\pm \left[z, \nu_u\right] \quad \text{$\mathcal{H}^{N-1}$-a.e. on $J_u$.}
		\end{equation}
	\end{lemma}
	We finish by giving properties of the spaces $DBV(\Omega)$ and $DTBV^{+}(\Omega)$. The next result is proven, for  $\alpha=1$, in \cite[Lemma 5.3]{GMP}.
	\begin{lemma}\label{Tu z parring}
		Let $u \in DTBV^+(\Omega) \cap L^\infty(\Omega)$ and $z \in \DM(\Omega)$. Then $z \chi_{\{u>a\}} \in \DM(\Omega)$ for almost every $a>0$ and \begin{equation}\label{Tu z parring eq}
			\left(z, D(T^{\infty}_a(u))^{\alpha}\right)=\left(z \chi_{\{u>a\}}, D(T^{\infty}_a(u))^{\alpha}\right) \quad \text{for a.e. $a>0$ and for all $\alpha>0$}.
		\end{equation}
	\end{lemma}
	\begin{proof}
		Let $\overline{T}_{a}(u)=T_{a}^{\infty}(u)^{\alpha}-a^{\alpha}$. Since $D(T_{a}^{\infty}(u))^{\alpha}=D\overline{T}_{a}(u)$ and since $\chi_{\{u>a\}}\in BV(\Omega)\cap L^\infty(\Omega)$ for almost every $a>0$\bk, we get
		\begin{align*}
			\left(z \chi_{\{u>a\}}, D(T^\infty_a(u))^{\alpha}\right)&=\left(z \chi_{\{u>a\}}, D\overline{T}_{a}(u)\right)\\
			&\stackrel{\eqref{= misure}}{=}\operatorname{div}\left(\overline{T}_{a}(u)z\chi_{\{u>a\}}\right)-\overline{T}_{a}(u)\operatorname{div}(z\chi_{\{u>a\}})\\
			&=\operatorname{div}\left(\overline{T}_{a}(u)z\right)-\overline{T}_{a}(u)\operatorname{div}(z\chi_{\{u>a\}})\\
			&\stackrel{\eqref{= misure}}{=}\operatorname{div}\left(\overline{T}_{a}(u)z\right)-\overline{T}_{a}(u)\operatorname{div}z-\overline{T}_{a}(u)(z,D\chi_{\{u>a\}})\\
			&=\operatorname{div}\left(\overline{T}_{a}(u)z\right)-\overline{T}_{a}(u)\operatorname{div}z= \left(z, D(T^{\infty}_a(u))^{\alpha}\right),
		\end{align*}	
		where we also used that  $ (z,D\chi_{\{u>a\}})\ll |D\chi_{\{u>a\}}|$  and the fact that $\overline{T}_{a}(u)=0\,|D\chi_{\{u>a\}}|$-a.e.  in $\Omega$ since $\mathcal{H}^{N-1}(S_u^*)=0$. \bk This concludes the proof.
	\end{proof}
	Finally, let us recall the following result which is proven in \cite[Lemma 2.6]{GMP}. \bk
	\begin{lemma}\label{l.DBV}
		Let $z \in \DM(\Omega)$ and let $u,v \in DBV(\Omega) \cap L^\infty(\Omega)$. Then
		\begin{equation}\label{uscire sx}
			\left(uz, Dv\right)=u\left(z, Dv\right),  
		\end{equation}
		and 		
		\begin{equation}\label{uscire prod}
			\left(z, D(uv)\right)=u\left(z, Dv\right)+v\left(z, Du\right)=\left(uz, Dv\right)+\left(vz, Du\right)
		\end{equation}
		as measures. 
	\end{lemma}
	
	\section{Statement of the main result for nonnegative data}\label{s.existence}
	Let us explain the concept of distributional solution for problem \eqref{prob} in case of a nonnegative datum. 
	\begin{defin}\label{def sol}
		Assume $m>0$ and $0\leq f\in L^{N,\infty}(\Omega)$. A nonnegative function  $u \in DTBV(\Omega) \cap L^\infty(\Omega)$ is a distributional solution to \eqref{prob}, if there exists a vector field $ w \in L^\infty(\Omega)^N$ such that $\|w\|_{L^\infty(\Omega)^N} \le 1$ and the vector field  $z:=u^m w \in \DM(\Omega)$ is such that
		\begin{equation}\label{sol 1}
			-\operatorname{div}z=f \quad \text{as measures in $\Omega$,}
		\end{equation}
		\begin{equation}\label{sol 2 b} 		
			\left(z, DT^\infty_a(u)\right) = \frac{1}{m+1}|DT^\infty_a(u)^{m+1}| \  \text{as measures in $\Omega$ for a.e. $a>0$,\bk}
		\end{equation}
		and
		\begin{equation}\label{sol 3b}
			\left[z, \nu\right]=-(u^{\Omega})^m \quad \text{$\mathcal{H}^{N-1}$-a.e. on $\partial \Omega \cap \{u>0\}$.}
		\end{equation}
	\end{defin}
	\begin{remark}
		Let us provide a more detailed explanation of the meaning behind Definition \ref{def sol}.
		It is worth noting that formula \eqref{sol 2 b}  illustrates the role of the vector field $w$ as the singular quotient $|Du|^{-1}Du$ and, similarly, the way the vector field $z$ assumes the role of  $u^m |Du|^{-1}Du$ in a weak sense.  We highlight that \eqref{sol 2 b} is equivalent to $$\left(z, DT^\infty_a(u)\right)=T^\infty_a(u)^m|DT^\infty_a(u)| \quad \text{as measures, for a.e.  $a>0$,}$$ as an application of \eqref{chain rule senza j} since $u \in DTBV(\Omega)$. 
		
		\smallskip 
		Finally it is worth mentioning that, as $u$ is not in  $BV$ up to the boundary of $\Omega$, in general one could suspect that condition \eqref{sol 3b} is not well defined; however,  we recall that Lemma \ref{lemma trace pos}, Definition \ref{defin set bd} and Lemma \ref{lemma well posed bd} ensure the existence of a trace for functions in $TBV(\Omega)$ and a meaning to $\{u>0\}\cap\partial\Omega$. 
	\end{remark}

	Now we state the main result of this section.
	\begin{theorem}\label{teo f N}
		Assume $m>0$ and  let  $0\le f \in L^{N,\infty}(\Omega)$. Then there exists a solution  $u$ to \eqref{prob}  in the sense of Definition  \ref{def sol}. In particular,  if $f \not \equiv 0$, then $u \not \equiv 0$.
	\end{theorem}
	\begin{remark}\label{rem_null}
		Let us stress that, in contrast with the case of the $1$-Laplacian (see for instance \cite{CT, MST1}), a solution in the sense of Definition \ref{def sol} can not be null once  $f \not \equiv 0$ as $z=u^m w$ and \eqref{sol 1} is in force.
		We also  emphasize that the existence of a non-trivial bounded  solution is obtained  regardless of any   smallness assumptions of $\|f\|_{L^{N,\infty}(\Omega)}$, again in contrast with the  $0$-homogeneous case, i.e.   $m=0$. 
		
	\end{remark}
	\section{Approximating problems and basic estimates}\label{s.approxsolutions}	
	
	\subsection{Existence for the perturbed problem} Following an  idea in \cite{GMP},  the proof of Theorem \ref{teo f N} will be performed by approximating \eqref{prob} with the smooth perturbed problem 
	\begin{equation}\label{e.problemapprox}
		\begin{cases}
			\dis-\operatorname{div}\left(|u|^{m}\frac{\nabla u}{|\nabla u|_\varepsilon}+\varepsilon \nabla u\right)= f &\text{in $\Omega$,} \\
			u =  0 &\text{ on 
				$\partial \Omega$},
		\end{cases}
	\end{equation}
	where $$|\xi|_{\varepsilon}=\sqrt{|\xi|^2+\varepsilon^2}\text{ for all }\xi\in \mathbb{R}^{N}.$$
	Notice,  in particular, that for any $\varepsilon>0$, one has 
	\begin{equation}\label{dis norm}
		\frac{|\xi|^2}{|\xi|_\varepsilon} \ge |\xi|-\varepsilon, \quad \text{for any $\xi \in \mathbb{R}^N$.}
	\end{equation} 
	Let us state and prove   the existence of a solution  for problem \eqref{e.problemapprox} in the general case of a datum in $H^{-1}(\Omega)$ (observe that $L^{N, \infty} (\Omega)$ embeds into $H^{-1}(\Omega)$).
	\begin{lemma}\label{l.lemmaapprox}
		If $f\in H^{-1}(\Omega)$, then for every $\varepsilon>0$ there exists a unique solution $u_{\varepsilon}\in H_{0}^{1}(\Omega)\cap L^{\infty}(\Omega)$ for the problem \eqref{e.problemapprox} in the sense that 
		\begin{equation}\label{e.weaksol}
			\int_{\Omega}\left(|u_\varepsilon|^m \frac{\nabla u_\varepsilon}{|\nabla u_\varepsilon|_\varepsilon}+\varepsilon \nabla u_\varepsilon\right) \cdot \nabla\varphi=\int_{\Omega}f\varphi \quad \text{ for all $\varphi\in H_{0}^{1}(\Omega)$.}
		\end{equation}  
		Moreover, if $f\geq0$, then $u_{\varepsilon}\geq0$.
	\end{lemma}
	\begin{proof} 
		The proof of this Lemma is a straightforward adaptation of \cite[Lemma 3.1]{GMP}, but for completeness, we sketch  the main steps.
		\\Let $\delta>0$ and  consider the following truncated problem 
		\begin{equation} \label{e.problemapproxaux}
			\begin{cases}
				-\operatorname{div}\left( A_{\delta}(u, \nabla u) \right)= f & \text{ in $\Omega$,} \\
				u =  0 & \text{on
					$\partial \Omega$,}
			\end{cases}
		\end{equation}
		where the operator $A_{\delta}:\mathbb{R}\times\mathbb{R}^{N} \to \mathbb{R}^N$ is defined by
		\begin{equation}\label{oper}
			A_{\delta}(s,\xi):=T_{1/\delta}(|s|)^m\frac{\xi}{|\xi|_\varepsilon}+\varepsilon  \xi \quad \text{for any $(s,\xi) \in   \mathbb{R} \times \mathbb{R}^N$,}
		\end{equation}
		with $T_{\frac{1}{\delta}}(s)$ being  defined in \eqref{troncata}.
		Problem \eqref{e.problemapproxaux} admits a unique solution, as guaranteed by \cite[Corollary 1]{B}. This result is attributed to the properties of the operator defined in formula \eqref{oper}, which are
		\begin{itemize}
			\item \textbf{Boundedness}
			$$
			|A_{\delta}(s,\xi)|\leq C+\varepsilon|\xi| \quad \text{for any $(s,\xi) \in   \mathbb{R} \times \mathbb{R}^N$,}
			$$
			where $C>0$ is a constant depending on $\varepsilon$, $\delta$ and $m$.
			\item \textbf{Monotonicity}
			\begin{equation*}
				(A_{\delta}(s,\xi_{1})-A_{\delta}(s,\xi_{2}))\cdot(\xi_{1}-\xi_{2})>0 \quad \text{ for any $(s, \xi_i) \in  \mathbb{R} \times \mathbb{R}^N$, with $i=1,2$ and $\xi_{1}\neq\xi_{2}$,}
			\end{equation*}
			which follows from the convexity of the associated Lagrangian $\mathcal{L}$ given by
			$$\mathcal{L}(s,\xi):=T_{1/\delta}( |s|\bk)^{m}\left(|\xi|^2 + \varepsilon^2\right)^{\frac{1}{2}}+\frac{\varepsilon}{2}|\xi|^2 \quad \text{for any $(s,\xi) \in  \mathbb{R} \times \mathbb{R}^N$.}$$
			\item \textbf{Coercivity}
			$$
			A_{\delta}(s,\xi)\cdot\xi \ge \varepsilon |\xi|^2 \quad \text{for any $(s , \xi ) \in   \mathbb{R} \times \mathbb{R}^N$.}$$
		\end{itemize}

		We  emphasize that we have found a solution, denoted as $u_{\varepsilon, \delta}$, which depends on both the parameters $\varepsilon$ and $\delta$.
		\\In particular, choosing $G_k(u_{\varepsilon, \delta})\in H^1_0(\Omega)$ as test function in the weak formulation of problem \eqref{e.problemapproxaux} yields
		$$\varepsilon \int_\Omega |\nabla G_k(u_{\varepsilon, \delta})|^2 \le \int_\Omega f G_k(u_{\varepsilon, \delta}).$$
		Applying Stampacchia's method (see \cite{S}), we get $$\|u_{\varepsilon, \delta}\|_{L^\infty(\Omega)} \le C,$$ where $C>0$ is a constant independent of $\delta$. Therefore picking out $\delta < \frac{1}{C}$, one can deduce that  
		$T_{\frac{1}{\delta}}(|u_{\varepsilon, \delta}|)=|u_{\varepsilon, \delta}|$. Thus the function $u_{\varepsilon,\delta}$, which we denote simply by $u_{\varepsilon}$, is the solution of problem \eqref{e.problemapprox} in the weak sense  \eqref{e.weaksol}. 
		Finally if $f \ge 0$,  one can take  $u_{\varepsilon}^{-}=\max\{-u_\varepsilon,0\}$ in \eqref{e.weaksol} as test function, yielding, after dropping a nonpositive term,  to $$ -\varepsilon\int_\Omega |\nabla u^-_\varepsilon|^2 \geq  \int_\Omega f u_\varepsilon^- \ge 0, $$ from which follows that $u_\varepsilon \ge 0$ and this concludes the proof. 
	\end{proof}

	\subsection{A priori estimates and existence of a limit function} In the following result we collect some a priori  estimates on $u_\varepsilon$ which ensure the existence of a limit function $u$ as $\varepsilon$ tends to zero. For the sake of completeness we state them in the slightly general case of datum $f\in L^{\widetilde{m}}(\Omega)$ with 
	\begin{equation}\label{mtilde}
		\widetilde{m}:=\frac{N(m+1)}{Nm+1} < N.  
	\end{equation}
	\begin{lemma}\label{l.stimeLq}
		Assume $m>0$, let $0\le f \in L^{\widetilde{m}}(\Omega)$ with $\widetilde{m}$ given in \eqref{mtilde},  and let $u_{\varepsilon}$ be the solution to \eqref{e.problemapprox}. Let  $C_{\varepsilon}$ be the constant given by\bk 
		\begin{equation}\label{ceps}
			C_{\varepsilon}:=\displaystyle \left(\frac{\|f\|_{L^{\widetilde{m}}(\Omega)}}{\frac{\mathcal{S}_{1}^{-1}}{m+1}-\frac{\varepsilon|\Omega|^{1-\frac{m}{(m+1)1^{*}}}}{\frac{1}{2}(\mathcal{S}_{1}(m+1)\|f\|_{L^{\widetilde{m}}(\Omega)})^{\frac{1}{m}}}}\right)^{\frac{1}{m}}.\bk 
		\end{equation}		
		Then it holds: there exists $\overline{\varepsilon}$ such that
		\begin{equation}\label{e.equexplt1}
			\|u_{\varepsilon}^{m+1}\|_{L^{1^{*}}(\Omega)}\leq C_{\varepsilon}^{m+1}\text{ for all }0<\varepsilon<\overline{\varepsilon},
		\end{equation}
		and
		\begin{equation}\label{stime approx BV}
			\|u_{\varepsilon}^{m+1}\|_{BV(\Omega)}\leq (m+1)\left(\varepsilon|\Omega|^{1-\frac{m}{(m+1)1^{*}}}C_{\varepsilon}^{m}+\|f\|_{L^{\widetilde{m}}(\Omega)}C_{\varepsilon}\right)\text{ for all }0<\varepsilon<\overline{\varepsilon}.
		\end{equation}
		In particular, $u_\varepsilon^{m+1} $ is uniformly bounded in $BV(\Omega)$ for any $0<\varepsilon<\overline{\varepsilon}$.
	\end{lemma}
	\begin{proof}
		Let us fix $\varphi= u_{\varepsilon}$ in \eqref{e.weaksol} obtaining
		$$
		\int_{\Omega}u_{\varepsilon}^{m}\frac{|\nabla u_{\varepsilon}|^2}{|\nabla u_{\varepsilon}|_\varepsilon}+\varepsilon\int_{\Omega}|\nabla u_{\varepsilon}|^2=\int_{\Omega}fu_{\varepsilon},
		$$
		from which, thanks \eqref{dis norm} and getting rid of the second nonnegative term, one yields to \bk
		$$
		\int_{\Omega}u_{\varepsilon}^{m}|\nabla u_{\varepsilon}|\leq\varepsilon\int_{\Omega}u_{\varepsilon}^{m}+\int_{\Omega}fu_{\varepsilon}.
		$$
		Since $u_{\varepsilon}\in H^1_0(\Omega)\cap L^{\infty}(\Omega)$ and since $s\mapsto s^{m+1}$ is locally Lipschitz, one has
		\begin{equation}\label{e.equ1}
			\frac{1}{m+1}\int_{\Omega}|\nabla u_{\varepsilon}^{m+1}|\leq\varepsilon\int_{\Omega}u_{\varepsilon}^{m}+\int_{\Omega}fu_{\varepsilon}.
		\end{equation}		
		Now we apply the Sobolev inequality on the left-hand of \eqref{e.equ1} and the H\"older inequality on the right-hand of \eqref{e.equ1}, yielding to 
		\begin{equation}\label{equ lemma 42}
			\frac{\mathcal{S}_{1}^{-1}}{m+1}\|u_{\varepsilon}\|^{m+1}_{L^{(m+1)1^{*}}(\Omega)}\leq\varepsilon|\Omega|^{1-\frac{m}{(m+1)1^{*}}}\|u_{\varepsilon}\|^{m}_{L^{(m+1)1^{*}}(\Omega)}+\|f\|_{L^{\widetilde{m}}(\Omega)}\|u_{\varepsilon}\|_{L^{(m+1)1^{*}}(\Omega)},
		\end{equation}
		where $\mathcal{S}_{1}$ is given by \eqref{e.Lstar} and $\widetilde{m}=\frac{N(m+1)}{Nm+1}$ (recall that  $\widetilde{m}< N$, so that $\|f\|_{L^{\widetilde{m}}(\Omega)}$ is finite). 
		
		\medskip 
		
		Now,   if 
		$$
		\|u_{\varepsilon}\|_{L^{(m+1)1^{*}}(\Omega)}\le(\mathcal{S}_{1}(m+1)\|f\|_{L^{\widetilde{m}}(\Omega)})^{\frac{1}{m}},
		$$
		definitively in $\varepsilon$ then the proof is concluded. Otherwise there exists a subsequence, which we still call $u_\varepsilon$, such that
		$$
		\|u_{\varepsilon}\|_{L^{(m+1)1^{*}}(\Omega)}\ge(\mathcal{S}_{1}(m+1)\|f\|_{L^{\widetilde{m}}(\Omega)})^{\frac{1}{m}},
		$$
		once again definitively in $\varepsilon$. Then, as $u_\varepsilon$ is not null, one has from \eqref{equ lemma 42} that it holds 
		\begin{equation*}\label{equ lemma 42bis}
			\frac{\mathcal{S}_{1}^{-1}}{m+1}\|u_{\varepsilon}\|^{m+1}_{L^{(m+1)1^{*}}(\Omega)}\leq \frac{\varepsilon|\Omega|^{1-\frac{m}{(m+1)1^{*}}}}{(\mathcal{S}_{1}(m+1)\|f\|_{L^{\widetilde{m}}(\Omega)})^{\frac{1}{m}}}\|u_{\varepsilon}\|^{m+1}_{L^{(m+1)1^{*}}(\Omega)}+\|f\|_{L^{\widetilde{m}}(\Omega)}\|u_{\varepsilon}\|_{L^{(m+1)1^{*}}(\Omega)},
		\end{equation*}
		from which one deduces the existence of $\overline{\varepsilon}$ such that
		$$
		\|u_{\varepsilon}\|_{L^{(m+1)1^{*}}(\Omega)}\leq \left(\frac{\|f\|_{L^{\widetilde{m}}(\Omega)}}{\frac{\mathcal{S}_{1}^{-1}}{m+1}-\frac{\varepsilon|\Omega|^{1-\frac{m}{(m+1)1^{*}}}}{\frac{1}{2}(\mathcal{S}_{1}(m+1)\|f\|_{L^{\widetilde{m}}(\Omega)})^{\frac{1}{m}}}}\right)^{\frac{1}{m}}\text{ for all }0<\varepsilon<\overline{\varepsilon},
		$$
		which is \eqref{e.equexplt1}. \bk
		Furthermore, the previous estimate, \eqref{e.equ1} and H\"older's inequality imply that $u_{\varepsilon}^{m+1}$ is  uniformly bounded in $BV(\Omega)$ with
		$$
		\|u_{\varepsilon}^{m+1}\|_{BV(\Omega)}\leq (m+1)\left(\varepsilon|\Omega|^{1-\frac{m}{(m+1)1^{*}}}C_{\varepsilon}^{m}+\|f\|_{L^{\widetilde{m}}(\Omega)}C_{\varepsilon}\right)\text{ for all }0<\varepsilon<\overline{\varepsilon}.
		$$
		This proves \eqref{stime approx BV} and it concludes the proof.   
	\end{proof}	
	
	\bk Next corollary gives the existence of a limit function $u$, to which $u_\varepsilon$ converges almost everywhere in $\Omega$.
	
	\begin{corollary}\label{cor_ex}
		Under the assumptions of Lemma \ref{l.stimeLq}, there exists a nonnegative $u \in TBV(\Omega)$ such that $u_\varepsilon^{m+1}$ converges to $u^{m+1}$ (up to subsequence) in $L^q(\Omega)$ for every $q < \frac{N}{N-1}$ and $Du_\varepsilon^{m+1}$ converges to $Du^{m+1}$ *-weakly as measures $\varepsilon$ tends to zero.
		Furthermore it holds
		\begin{equation}\label{stime BV}
			\|u^{m+1}\|_{BV(\Omega)}\leq\mathcal{S}_{1}^{\frac{1}{m}}((m+1)\|f\|_{L^{\widetilde{m}}(\Omega)})^{\frac{m+1}{m}},
		\end{equation}
		and
		\begin{equation}\label{stime star}
			\|u^{m+1}\|_{L^{1^{*}}(\Omega)}\leq\left(\mathcal{S}_{1}(m+1)\|f\|_{L^{\widetilde{m}}(\Omega)}\right)^{\frac{m+1}{m}}.
		\end{equation}
		In particular, if $f\equiv0$ then $u\equiv0$.
	\end{corollary}
	\begin{proof}
		\medskip
		
		By appealing to  Lemma \ref{l.stimeLq},   the compactness of the embedding $BV(\Omega)\hookrightarrow L^{r}(\Omega)$ with $1\leq r<1^{*}$ implies that there exists $v \in BV(\Omega)$ such that
		\begin{equation}\label{e.convBV}
			Du_\varepsilon^{m+1}\rightharpoonup D v \ \text{ *-weakly in }\mathcal{M}(\Omega),
		\end{equation}
		\begin{equation}\label{e.convLqm}
			u_{\varepsilon}^{m+1}\to v \text{ strongly in }L^{r}(\Omega), \ \ 1\leq r<1^{*},
		\end{equation}
		and
		\begin{equation}\label{e.quasiuvunquem}
			u_{\varepsilon}^{m+1}\to v \text{ almost everywhere in }\Omega.
		\end{equation}
		From Lemma \ref{l.lemmaapprox}, we know that $u_{\varepsilon}\geq0$ for all $0<\varepsilon<1$, so that $v\geq0$. As a consequence, we may define
		$$\label{def di u}
		u:=v^{\frac{1}{m+1}}.
		$$
		Using \eqref{e.quasiuvunquem}, we get
		$$\label{e.quasiovunqque}
		u_{\varepsilon}\to u\text{ almost everywhere in }\Omega.
		$$
		By \eqref{e.convLqm} and the Lebesgue Theorem, we conclude that 
		$$\label{e.convLq}
		u_{\varepsilon}\to u \quad \text{strongly in $L^q(\Omega)$, for all $1 \le q < 1^*(m+1)$,}
		$$
		and, in particular 
		\begin{equation}\label{e.convmL1}
			u_{\varepsilon}^{m}\to u^{m}  \quad \text{strongly in $L^q(\Omega)$, for all $1 \le q < \frac{1^*(m+1)}{m}$.}
		\end{equation}
		We now observe that $u$ does not necessarily belong to $BV(\Omega)$. Nevertheless, since $s \mapsto s^{\frac{1}{m+1}} \in C^\infty(a, \infty)$, for every $a>0$,  we can apply Theorem \ref{t chain rule} to obtain that $u \in TBV(\Omega)$.
		
		We conclude by showing \eqref{stime BV} and \eqref{stime star}. Expression \eqref{stime approx BV}, \eqref{e.convBV} and \eqref{sci BV} imply that
		\begin{align*}
			\|u^{m+1}\|_{BV(\Omega)}\leq\liminf_{\varepsilon\to0}\|u_{\varepsilon}^{m+1}\|_{BV(\Omega)}&\leq(m+1)\|f\|_{L^{\widetilde{m}}(\Omega)}\lim_{\varepsilon\to0} C_{\varepsilon}\\
			&=(m+1)\|f\|_{L^{\widetilde{m}}(\Omega)}(\mathcal{S}_{1}(m+1)\|f\|_{L^{\widetilde{m}}(\Omega)})^{\frac{1}{m}}.
		\end{align*}
		This proves \eqref{stime BV}. Similarly, using \eqref{e.equexplt1}
		$$
		\|u^{m+1}\|_{L^{1^*}(\Omega)}\leq\liminf_{\varepsilon\to0}\|u_{\varepsilon}^{m+1}\|_{L^{1^*}(\Omega)}\leq\lim_{\varepsilon\to0}C_{\varepsilon}^{m+1}=(\mathcal{S}_{1}(m+1)\|f\|_{L^{\widetilde{m}}(\Omega)})^{\frac{m+1}{m}},
		$$
		which shows \eqref{stime star}. 
	\end{proof}
	
	Now we show that $u$ is also bounded.
	
	\begin{lemma}\label{lemma bound} Under the assumptions of Lemma \ref{l.stimeLq}, let $u$ be the function defined in Corollary \ref{cor_ex}. Then it holds 
		\begin{equation}\label{e.boundexplLinfty}
			\|u\|_{L^{\infty}(\Omega)}\leq(\widetilde{\mathcal{S}}_1\|f\|_{L^{N,\infty}(\Omega)})^{\frac{1}{m}}.	
		\end{equation}
	\end{lemma}
	\begin{proof}
		
		Let $k>0$ and take $\varphi=G_{k}(u_{\varepsilon})$ in \eqref{e.weaksol}, where $G_k$ is the function defined in \eqref{tronc 2}.  We get
		$$
		\int_{\Omega}u_{\varepsilon}^{m}\frac{|\nabla G_{k}(u_{\varepsilon})|^2}{|\nabla G_{k}(u_{\varepsilon})|_{\varepsilon}}\leq\int_{\Omega}fG_{k}(u_{\varepsilon}).
		$$
		From \eqref{dis norm} and from the definition of $G_{k}$, we obtain
		$$
		\int_{\Omega}|\nabla G_{k}(u_{\varepsilon})|\leq\varepsilon|A_{k}|+\frac{1}{k^{m}}\int_{\Omega}fG_{k}(u_{\varepsilon}) \stackrel{\eqref{holder lor}}{\le} \varepsilon|A_{k}|+\frac{1}{k^{m}}\|f\|_{L^{N,\infty}(\Omega)}\|G_{k}(u_{\varepsilon})\|_{L^{1^*,1}(\Omega)},
		$$
		where $A_{k}=\{u_{\varepsilon}>k\}$. Using \eqref{des lorentz}, for $k$ large enough, we get 
		$$
		\int_{\Omega}|\nabla G_{k}(u_{\varepsilon})|\leq\frac{\varepsilon|A_{k}|}{1-\frac{\widetilde{\mathcal{S}}_1 \|f\|_{L^{N,\infty}(\Omega)}}{k^{m}}}.
		$$
		Thanks to \eqref{e.Lstar}, we obtain
		$$
		\|G_{k}(u_{\varepsilon})\|_{L^{1^{*}}(\Omega)}\leq\frac{\varepsilon|A_{k}|}{\mathcal{S}_{1}^{-1}-\frac{\widetilde{\mathcal{S}}_1\mathcal{S}_{1}^{-1}\|f\|_{L^{N,\infty}(\Omega)}}{k^{m}}}.
		$$
		Now, for fixed  $0<\tau<1$, let $ k> k_{0,\tau}>0$ with $k_{0,\tau}$ given by  
		\begin{equation}\label{e.kappazerochoice}
			1-\frac{\widetilde{\mathcal{S}}_1 \|f\|_{L^{N,\infty}(\Omega)}}{k_{0,\tau}^{m}}=\tau.
		\end{equation}
		Thus 
		$$
		\int_{\Omega} |G_{k}(u_{\varepsilon})|^{1^{*}}\leq\left(\frac{ \mathcal{S}_{1}\varepsilon}{\tau}\right)^{1^{*}}|A_{k}|^{1^{*}}\text{ for all }k>k_{0,\tau}.
		$$
		On the other hand, we know that $A_{h}\subset A_{k}$ and that $G_{k}(u_{\varepsilon})\geq h-k$ in $A_{h}$ for each $h>k>k_{0,\tau}$. Consequently,
		\begin{equation}\label{e.levelset}
			|A_{h}|\leq\frac{1}{(h-k)^{1^{*}}}\left(\frac{ \mathcal{S}_{1}\varepsilon}{\tau}\right)^{1^{*}}|A_{k}|^{1^{*}},\text{  for all }h>k>k_{0,\tau}.
		\end{equation}
		As we are interested in the explicit $L^{\infty}$ bound we recall 	the classical Stampacchia's argument that runs as  follows: let $\psi_{k_{0}}(s):=|\{u_{\varepsilon}-k_{0,\tau}>s\}|$. Inequality \eqref{e.levelset} then becomes
		$$
		\psi_{k_{0,\tau}}(t)\leq\frac{1}{(t-s)^{1^{*}}}\left(\frac{ \mathcal{S}_{1}\varepsilon}{\tau}\right)^{1^{*}}\psi_{k_{0,\tau}}(s)^{1^{*}},\text{  for all }t>s>0.
		$$	
		From \cite[Lemma 4.1]{S}  
		$$
		|\psi_{k_{0,\tau}}(d)|=0,\ \ \text{ with } d=\left(\frac{ \mathcal{S}_{1}\varepsilon}{\tau}\right)|A_{k_{0,\tau}}|^{\frac{1}{N}}2^{N}.
		$$
		Consequently,
		\begin{equation}\label{e.Linftybound}
			\|u_{\varepsilon}\|_{L^{\infty}(\Omega)}\leq k_{0,\tau}+\left(\frac{ \mathcal{S}_{1}\varepsilon}{\tau}\right) 2^N|\Omega|^{\frac{1}{N}}\text{ for all }0<\varepsilon, \tau<1.
		\end{equation}
		Letting $\varepsilon\to0$ in \eqref{e.Linftybound}, we get 
		\begin{equation}\label{e.proofbounddinfty}
			\|u\|_{L^{\infty}(\Omega)}\leq k_{0,\tau}=\left(\frac{\widetilde{\mathcal{S}}_1\|f\|_{L^{N,\infty}(\Omega)}}{1-\tau}\right)^{\frac{1}{m}}\text{ for all }0< \tau<1.
		\end{equation}
		Then \eqref{e.boundexplLinfty}   follows by letting $\tau\to0$ in \eqref{e.proofbounddinfty}.	
		
	\end{proof}
	\begin{remark}
		The fact that $m>0$ is crucial in the proof of Lemma \ref{l.stimeLq}; if $m=0$, then  a condition on the size of $\|f\|_{L^{N,\infty}(\Omega)}$ would appear as expected, (to be compared with \eqref{e.kappazerochoice}).
		Let us highlight that as $m \to 0$ and $\widetilde{\mathcal{S}}_1 \|f\|_{L^{N,\infty}(\Omega)}<1$ from \eqref{e.boundexplLinfty} we obtain the classical result, i.e. $u \equiv 0$, as prove in \cite{CT}.

	\end{remark}

	\section{Proof of Theorem \ref{teo f N}}\label{s.proof}
	This section is devoted to show that $u$, which has been identified in Corollary \ref{cor_ex}, is a solution to \eqref{prob} in the sense of Definition \ref{def sol}. We proceed step by step by splitting the proof of Theorem \ref{teo f N} into five lemmas. First we show the existence of the vector field   $z$ satisfying \eqref{sol 1}.   
	\begin{lemma}\label{l.weaksol}
		Assume $m>0$ and let $0\le f \in L^{N,\infty}(\Omega)$. Let $u$ be the function identified in Corollary \ref{cor_ex}. Then there exists a vector field $w \in L^\infty(\Omega)^N$ with $\|w\|_{L^\infty(\Omega)^N} \leq 1$ such that  $z:=u^m w\in \DM(\Omega)$ satisfies 
		\begin{equation}\label{lem_sol 1}
			-\operatorname{div}z=f \quad \text{as measures in $\Omega$}. 
		\end{equation}
		Furthermore $u \not \equiv 0$ if $f \not \equiv 0$.
	\end{lemma}
	\begin{proof} 	Let $u_{\varepsilon}$ be the solution to \eqref{e.problemapprox}. Let us define
		$$
		w_{\varepsilon}=\frac{\nabla u_\varepsilon}{|\nabla u_{\varepsilon}|_{\varepsilon}},
		$$
		and note that $w_{\varepsilon}$ is uniformly bounded in $L^{\infty}(\Omega)^N$  as
		$\|w_{\varepsilon}\|_{L^{\infty}(\Omega)^N}\leq 1$. Therefore there exists $w \in L^\infty(\Omega)^N$ with $\|w\|_{L^\infty(\Omega)^N} \le 1$ such that $w_\varepsilon$ converges *-weakly to $w$ in $L^\infty(\Omega)^N$ as $\varepsilon\to 0$; in particular it holds 
		\begin{equation}\label{e.weakz}
			\lim_{\varepsilon\to0}\int_{\Omega}w_{\varepsilon}\cdot\Psi=
			\int_{\Omega}w\cdot\Psi\text{ for all }\Psi\in L^{r}(\Omega)^{N},\  r\geq 1.
		\end{equation}
		Now we aim to take $\varepsilon\to 0$ into \eqref{e.weaksol}. On one hand, it holds
		$$
		\lim_{\varepsilon\to0}\int_{\Omega}u_{\varepsilon}^{m}w_{\varepsilon}\cdot\nabla\varphi=\int_{\Omega}u^{m}w\cdot\nabla\varphi\text{ for all }\varphi\in C_{c}^{1}(\Omega),
		$$
		since, from Corollary \ref{cor_ex}, $u^m_\varepsilon \to u^m$ strongly in $L^1(\Omega)$ and since we have just shown that $w_\varepsilon \rightharpoonup w$ *-weakly in $L^\infty(\Omega)^N$. In particular, from now on, we define  $z:=u^m w$. 
		
		On the other hand, by taking $\varphi= u_{\varepsilon}$ in \eqref{e.weaksol}, dropping a positive term,  and by using Young's inequality \eqref{holder lor}, we get
		\begin{equation}\label{passlimit}
			\varepsilon\int_{\Omega}|\nabla u_{\varepsilon}|^2\leq\int_{\Omega}fu_{\varepsilon}\leq C\|f\|_{L^{N,\infty}(\Omega)},\text{ for all $0<\varepsilon<1$},
		\end{equation}
		where $C>0$ does not depend on $\varepsilon$ thanks to \eqref{e.Linftybound}. Consequently, using the H\"older inequality
		
		\begin{equation}\label{e.limzero}
			\begin{aligned}
				\biggl|\varepsilon\int_{\Omega}\nabla u_{\varepsilon}\cdot\nabla\varphi\biggl|&\leq\varepsilon^{1/2}\left(\int_{\Omega}\varepsilon|\nabla u_{\varepsilon}|^{2}\right)^{\frac{1}{2}}\left(\int_{\Omega}|\nabla \varphi|^{2}\right)^{\frac{1}{2}}
				\\ 
				&\stackrel{\eqref{passlimit}}{\le} \varepsilon^{\frac{1}{2}}C^{\frac{1}{2}}\|f\|^{\frac{1}{2}}_{L^{N,\infty}(\Omega)}\left(\int_{\Omega}|\nabla \varphi|^{2}\right)^{\frac{1}{2}},
			\end{aligned}
		\end{equation}
		which converges to $0$ as $\varepsilon \to 0$ for all $\varphi\in C_{c}^{1}(\Omega)$.
		This shows the validity of \eqref{lem_sol 1} which also implies that $u \not \equiv 0$ if $f \not \equiv 0$ (see also Remark \ref{rem_null}). This proves the result. 
	\end{proof}

	\medskip
	
	\subsection{The identification of the vector field $z$} First we prove the following lemma.
	
	\begin{lemma}\label{l.totvar}
		Assume $m>0$ and $0\le f \in L^{N,\infty}(\Omega)$. Let $u$ be  the function defined in Corollary \ref{cor_ex} and let $z$ be the vector field defined in Lemma \ref{l.weaksol}. Then  
		\begin{equation}\label{e.totvar}
			\left(z,DT_{a}^{\infty}(u)\right)\geq \frac{1}{m+1}|DT_{a}^{\infty}(u)^{m+1}| \ \text{as measures for all $a>0.$}	
		\end{equation}
	\end{lemma}\bk

	\begin{proof} 
		Let $u_\varepsilon$ be the solutions of \eqref{e.weaksol} and let us take $T_{a}^{\infty}(u_{\varepsilon})\varphi$ as a test function in \eqref{e.weaksol} with $0\le \varphi\in C_{c}^{1}(\Omega)$. This yields to 	$$
		\int_{\Omega}u_{\varepsilon}^{m}w_{\epsilon}\cdot\nabla T_{a}^{\infty}(u_{\varepsilon})\,\varphi+\int_{\Omega}u_{\varepsilon}^{m}w_{\varepsilon}\cdot\nabla \varphi \,T_{a}^{\infty}(u_{\varepsilon})+\varepsilon\int_{\Omega}\nabla u_{\varepsilon}\cdot\nabla (T_{a}^{\infty}(u_{\varepsilon})\varphi)=\int_{\Omega}f T_{a}^{\infty}(u_{\varepsilon})\varphi.
		$$
		Notice that
		$$
		\varepsilon\int_{\Omega}\nabla u_{\varepsilon}\cdot\nabla (T_{a}^{\infty}(u_{\varepsilon})\varphi)=\varepsilon\int_{\Omega}|\nabla u_{\varepsilon}|^2 (T_{a}^{\infty})^{\prime}(u_{\varepsilon})\varphi+\varepsilon\int_{\Omega}\nabla u_{\varepsilon}\cdot\nabla\varphi\,T_{a}^{\infty}(u_{\varepsilon})\geq\varepsilon\int_{\Omega}\nabla u_{\varepsilon}\cdot\nabla\varphi\, T_{a}^{\infty}(u_{\varepsilon}),
		$$
		which gives
		\begin{equation}\label{e.proof2l.totvar}
			\int_{\Omega}u_{\varepsilon}^{m}w_{\varepsilon}\cdot\nabla T_{a}^{\infty}(u_{\varepsilon})\,\varphi+\int_{\Omega}u_{\varepsilon}^{m}w_{\varepsilon}\cdot\nabla \varphi\, T_{a}^{\infty}(u_{\varepsilon})+\varepsilon\int_{\Omega}\nabla u_{\varepsilon}\cdot\nabla \varphi\, T_{a}^{\infty}(u_{\varepsilon})\leq\int_{\Omega}f T_{a}^{\infty}(u_{\varepsilon})\varphi.
		\end{equation}
		Now observe that 
		\begin{equation}\label{e.proof3l.totvar}
			\lim_{\varepsilon\to 0}\varepsilon\int_{\Omega}\nabla u_{\varepsilon}\cdot\nabla \varphi\, T_{a}^{\infty}(u_{\varepsilon})=0,
		\end{equation}
		which can be shown as for \eqref{e.limzero} since  $T_{a}^{\infty}(u_{\varepsilon})$ is bounded in $L^{\infty}(\Omega)$ by Lemma \ref{lemma bound}.
		Next we show that
		\begin{equation}\label{e.proof4l.totvar}
			\lim_{\varepsilon\to 0}\int_{\Omega}u_{\varepsilon}^{m}w_{\varepsilon}\cdot\nabla \varphi\, T_{a}^{\infty}(u_{\varepsilon})=\int_{\Omega}z\cdot\nabla\varphi\, T_{a}^{\infty}(u).
		\end{equation}
		First observe that the Lebesgue Theorem, \eqref{e.quasiuvunquem} and \eqref{e.Linftybound} imply that  
		$u_{\varepsilon}^{m} T_{a}^{\infty}(u_{\varepsilon})\to u^{m} T_{a}^{\infty}(u)$ in $L^{1}(\Omega)$. This,   \eqref{e.weakz},  and the fact that $\|w_{\varepsilon}\|_{L^{\infty}(\Omega)^N}\leq1$ give  
		\begin{align*}
			u_{\varepsilon}^{m}w_{\varepsilon}\cdot\nabla \varphi \, T_{a}^{\infty}(u_{\varepsilon})-z\cdot\nabla\varphi\, T_{a}^{\infty}(u)&=(T_{a}^{\infty}(u_{\varepsilon})u_{\varepsilon}^{m}-T_{a}^{\infty}(u)u^{m})w_{\varepsilon}\cdot\nabla\varphi\\
			&+T_{a}^{\infty}(u)u^{m}(w_{\varepsilon}-w)\cdot\nabla\varphi\to0\text{ in }L^{1}(\Omega).
		\end{align*}
		This proves \eqref{e.proof4l.totvar}. 
		
		\smallskip 
		Furthermore, by computing $\nabla T_{a}^{\infty}(u_{\varepsilon})$ in terms of $\nabla u_{\varepsilon}$ and by using \eqref{dis norm}, we get
		\begin{equation}\label{e.proof5l.totvar}
			\int_{\Omega}u_{\varepsilon}^{m}w_{\varepsilon}\cdot\nabla T_{a}^{\infty}(u_{\varepsilon})\,\varphi\geq\int_{\Omega}u_{\varepsilon}^{m}|\nabla u_{\varepsilon}|(T_{a}^{\infty})'(u_{\varepsilon})\varphi-\varepsilon\int_{\Omega}u_{\varepsilon}^{m}(T_{a}^{\infty})'(u_{\varepsilon})\varphi.	
		\end{equation}
		Substituting \eqref{e.proof3l.totvar}, \eqref{e.proof4l.totvar} and \eqref{e.proof5l.totvar} in \eqref{e.proof2l.totvar}, we obtain
		\begin{equation}\label{e.proof6l.totvar} 
			\small{\limsup_{\varepsilon\to0}\left(\int_{\Omega}u_{\varepsilon}^{m}|\nabla u_{\varepsilon}| (T_{a}^{\infty})'(u_{\varepsilon})\varphi- \varepsilon\int_{\Omega} u_{\varepsilon}^{m} (T_{a}^{\infty})'(u_{\varepsilon})\varphi \right)+\int_{\Omega}z\cdot\nabla\varphi T_{a}^{\infty}(u)\leq\int_{\Omega}fT_{a}^{\infty}(u)\varphi.}
		\end{equation}
		We now estimate the terms in \eqref{e.proof6l.totvar}. From \eqref{e.convmL1}, it is clear that 
		\begin{equation}\label{e.proof7}
			\lim_{\varepsilon\to0}\varepsilon\int_{\Omega}u_{\varepsilon}^{m}(T_{a}^{\infty})'(u_{\varepsilon})\varphi=\lim_{\varepsilon \to 0}\varepsilon\int_{\{u_\varepsilon>a\}} u^m_\varepsilon \varphi=0,
		\end{equation} 
		Also, computing $\nabla T_{a}^{\infty}(u_{\varepsilon})$ yields
		\begin{equation}\label{e.proof8}
			\int_{\Omega}u_{\varepsilon}^{m}|\nabla u_{\varepsilon}|(T_{a}^{\infty})'(u_{\varepsilon})\varphi=\frac{1}{m+1}\int_{\Omega}|\nabla (T_{a}^{\infty}(u_{\varepsilon}))^{m+1}|\varphi.
		\end{equation}
		By substituting \eqref{e.proof7} and \eqref{e.proof8} in \eqref{e.proof6l.totvar}, and by using \eqref{sci BV0}, we then  get
		$$
		\frac{1}{m+1}\int_{\Omega}|D( T_{a}^{\infty}(u)^{m+1})|\varphi+\int_{\Omega}z\cdot\nabla\varphi\, T_{a}^{\infty}(u)\leq\int_{\Omega}fT_{a}^{\infty}(u)\varphi.
		$$
		Then, using \eqref{lem_sol 1}  and \eqref{pairing}, one gains
		$$\frac{1}{m+1}\int_{\Omega}|DT^\infty_a(u)^{m+1}| \varphi \le \int_{\Omega}(z, DT^\infty_a(u)) \varphi,$$
		which concludes the proof. 
	\end{proof}
	As we aim to show that equality holds in \eqref{e.totvar}, we need first  to show  that $u$ does not admit jumps.
	\begin{lemma}\label{u reg}
		Assume $m>0$ and $0\le f \in L^{N,\infty}(\Omega)$. Let $u$ be the function defined in Corollary \ref{cor_ex}. Then $u \in DTBV^+(\Omega) \cap L^\infty(\Omega)$. \bk
	\end{lemma}
	\begin{proof}
		We follow the argument of \cite[Lemma 5.9]{GMP};  in order to ease the presentation we split the proof in few steps.
		
		\medskip
		{\bf Step 1.} Let us start with some preliminary remarks. First observe that from  Corollary \ref{cor_ex} and Lemma \ref{lemma bound} we have that $u\in TBV(\Omega)\cap L^{\infty}(\Omega)$, and this allows us to deduce that, for fixed $a>0$,  both  $T^\infty_a(u)$  and  $T^\infty_a(u)^{m+1}$ belong to $ BV(\Omega)\cap L^{\infty}(\Omega)$. 

		\smallskip By \cite[Proposition 3.69]{AFP} we know that $$J_{T^\infty_a(u)}=J_{T^\infty_a(u)^{m+1}},$$ and, moreover, $\nu_{T^\infty_a(u)}=\nu_{T^\infty_a(u)^m}$ on $J_{T^\infty_a(u)}$ for almost every $a>0$ since $m>0$.   Roughly speaking  we will show that $\mathcal{H}^{N-1}\left(J_{T^\infty_a(u)^{m+1}}\right)=0$ and so  $\mathcal{H}^{N-1}\left(J_{T^\infty_a(u)}\right)=0$. 
This fact, in view of Lemma \ref{salti in TBV},   will follow once we prove that $\mathcal{H}^{N-1}\left(S_{u}^{*}\right)=0$.

\medskip		
	{\bf Step 2.}	 Let $z$ be the vector field found in Lemma \ref{l.weaksol}. 	
By Lemma \ref{l.weaksol}   equation \eqref{lem_sol 1} holds; this implies in particular that  $\operatorname{div}z \in L^{N,\infty}(\Omega)$.

		Thus, thanks to Lemma \ref{lem sigma}, one gets  \begin{equation}\label{saltino nullo}
			0=\operatorname{div}z\res J_{T^\infty_a(u)}= \left(\left[z, \nu_{T^\infty_a(u)}\right]^+ - \left[z, \nu_{T^\infty_a(u)}\right]^-\right) \mathcal{H}^{N-1}\res J_{T^\infty_a(u)}.
		\end{equation} 
		As a consequence 
		\begin{equation}\label{def psi}
			\left[z, \nu_{T^\infty_a(u)}\right]:=\left[z, \nu_{T^\infty_a(u)}\right]^+ =\left[z, \nu_{T^\infty_a(u)}\right]^- \, \text{$\mathcal{H}^{N-1}$-a.e. on $J_{T^\infty_a(u)}$.}
		\end{equation}
		
		\medskip
		{\bf Step 3.}
		 We claim  that $w\chi_{\{u>a\}}\in \DM(\Omega)$ for almost every $a>0$.  As $z\in \DM(\Omega)$, this follows by Lemma \ref{campo prodotto} as $w\chi_{\{u>a\}}=u^{-m}z\chi_{\{u>a\}},
		$
		and $u^{-m}\chi_{\{u>a\}}\in BV(\Omega)\cap L^{\infty}(\Omega)$, since $u\in TBV(\Omega)\cap L^{\infty}(\Omega)$.

		Through  \eqref{saltando sul bordo}, for almost every $a>0$,   also recalling \eqref{riferz},  we get 
		{\small$$\small \left|\left[z, \nu_{T^\infty_a(u)}\right]\bk \chi_{\{u>a\}} \right|= \left|(T^\infty_a(u)^m)^\pm \left[w \chi_{\{u>a\}}, \nu_{T^\infty_a(u)}\right]^\pm\right| \le \left(T^\infty_a(u)^m\right)^\pm \, \text{$\mathcal{H}^{N-1}$-a.e. on $J_{T^\infty_a(u)} $,}$$}
		which implies that , for almost every $a>0$, 
		\begin{equation}\label{dis per psi}
			\left|\left[z, \nu_{T^\infty_a(u)}\right] \chi_{\{u>a\}}\right| \le \min\left\{\left(T^\infty_a(u)^m\right)^+, \left(T^\infty_a(u)^m\right)^-\right\} \ \ \text{$\mathcal{H}^{N-1}$-a.e. on $J_{T^\infty_a(u)} $.}
		\end{equation}
		
		\medskip
		
		{\bf Step 4.}
		 On    $J_{T^\infty_a(u)}$ one has the following inequality:  
		\begin{equation}\label{salto di s^m+1}
			\begin{aligned}
				\frac{1}{m+1}\left|D^j T^\infty_a(u)^{m+1}\right|   \stackrel{\eqref{e.totvar}}{\le} &\left(z, DT^\infty_a(u)\right) \res J_{T^\infty_a(u)} \\ \stackrel{\eqref{= misure}}{=}  &\left(- T^\infty_a(u)^* \operatorname{div}z+ \operatorname{div}\left(T^\infty_a(u)z\right) \right)\res J_{T^\infty_a(u)} 
				\\
				\stackrel{\eqref{saltino nullo}}{=} &\operatorname{div}\left(T^\infty_a(u)z\right) \res J_{T^\infty_a(u)}.
			\end{aligned}
		\end{equation}

	 \medskip
Therefore, recalling Theorem \ref{t chain rule} and also using \eqref{def psi}, we have
		\begin{align*}
			\frac{1}{m+1}&\left|\left(T^\infty_a(u)^{m+1}\right)^+ - \left(T^\infty_a(u)^{m+1}\right)^-\right| \mathcal{H}^{N-1}\res J_{T^\infty_a(u)} \stackrel{\eqref{salto di s^m+1}}{\le} \operatorname{div}\left(T^\infty_a(u)z\right)   \res J_{T^\infty_a(u)} \\ = &\left(\left[T^\infty_a(u)z, \nu_{T^\infty_a(u)}\right]^+ - \left[T^\infty_a(u)z, \nu_{T^\infty_a(u)}\right]^-\right) \mathcal{H}^{N-1} \res J_{T^\infty_a(u)} \\ \stackrel{\eqref{saltando sul bordo}}{=}& \left(T^\infty_a(u)^+\left[z, \nu_{T^\infty_a(u)}\right]^+ - T^\infty_a(u)^-\left[z, \nu_{T^\infty_a(u)}\right]^-\right) \mathcal{H}^{N-1} \res J_{T^\infty_a(u)} \\ \stackrel{\eqref{def psi}}{=}& \left(T^\infty_a(u)^+ - T^\infty_a(u)^-\right)  \left[z, \nu_{T^\infty_a(u)}\right] \chi_{\{u>a\}}  \mathcal{H}^{N-1} \res J_{T^\infty_a(u)} \\ \stackrel{\eqref{dis per psi}}{\le} & \left|T^\infty_a(u)^+ - T^\infty_a(u)^-\right| \min \left\{\left(T^\infty_a(u)^m\right)^+, \left(T^\infty_a(u)^m\right)^-\right\} \mathcal{H}^{N-1} \res J_{T^\infty_a(u)}.
		\end{align*}
		Hence, as $\psi(s)=s^m$ is strictly monotone, one gets that 
		$$0=\mathcal{H}^{N-1}\left(J_{T^\infty_a(u)}\right)=\mathcal{H}^{N-1}(S_{T^\infty_a(u)}) \quad \text{for almost every $a>0$,}$$
		which implies, due to Lemma \ref{salti in TBV}, that $\mathcal{H}^{N-1}(S^*_u)=0$.
		This concludes the proof.
	\end{proof}

	Now we show that inequality \eqref{e.totvar}  obtained in Lemma \ref{l.totvar} is actually an equality.
	\begin{lemma}\label{c.corototvar}
		Assume $m>0$ and let $0\le f \in L^{N,\infty}(\Omega)$. Let  $u$ be the function defined in Corollary \ref{cor_ex} and let $z$ be the vector field defined in Lemma \ref{l.weaksol}. Then it holds
		\begin{equation}\label{sol 2 blem}
			\left(z, DT^\infty_a(u)\right) = \frac{1}{m+1}|DT^\infty_a(u)^{m+1}| \ \text{as measures in $\Omega$ for a.e. $a>0$.}
		\end{equation}
	\end{lemma}\bk
	\begin{proof} The proof strictly follows the one of \cite[Lemma 5.10]{GMP}, but for the sake of completeness, we present the details.

		One has 
		\begin{align}
			\nonumber	\frac{1}{m+1}|DT^\infty_a(u)^{m+1}|  \stackrel{\eqref{e.totvar}}{\le} &\left(z, DT^\infty_a(u)\right) \stackrel{\eqref{Tu z parring eq}}{=} \left(z \chi_{\{u>a\}}, DT^\infty_a(u)\right)  \\   \stackrel{\eqref{uscire sx} }{=} & T^\infty_a(u)^m \left(w \chi_{\{u>a\}}, DT^\infty_a(u)\right)  \label{516}
			\\   \stackrel{\eqref{ec:2}}{\leq }  &T^\infty_a(u)^m|DT^\infty_a(u)| \nonumber\\    \stackrel{\eqref{chain rule senza j}}{=}& \frac{1}{m+1}|DT^\infty_a(u)^{m+1}|\nonumber, 
		\end{align}
		where   in the last equality, we utilized the regularity property of $u$, which belongs to $DTBV^+(\Omega)$. Therefore,  we deduce that    \eqref{sol 2 blem} holds true. This completes the proof. 
	\end{proof}
	\begin{remark} \label{c.corototvarr}
		Observe that, by \eqref{516},  one also obtains $$
		\left(w \chi_{\{u>a\}}, DT^\infty_a(u)\right)=\left|D T^\infty_a(u)\right|  \quad \text{in $\mathcal{D}'(\Omega)$, for a.e. $a>0$,}
		$$ 
		which could be used as an equivalent condition in place of \eqref{sol 2 b} in Definition \ref{def sol}.
	\end{remark}
	\begin{remark}\label{ext ft} We highlight that,   thanks to \cite[Corollary 3.5]{LTS}, we can extend the space of test functions in \eqref{sol 1} to $BV(\Omega)\cap L^\infty(\Omega)$. Hence, using   \eqref{int per parti}, we can recast it as 
		\begin{equation}\label{56}\int_\Omega (z, D\psi)- \int_{\partial \Omega} \psi [z,\nu]=\int_{\Omega}f \psi \quad \text{for all $\psi \in BV(\Omega)\cap L^\infty(\Omega)$.}\end{equation}
	\end{remark}

	\medskip
	\subsection{The boundary condition} Here we show that $u$  satisfies  the boundary datum in the weak  sense given by  (\ref{sol 3b}). The key technical lemma is the following one. It is a step by step  re-adaptation of the proof of formula  (4.15) in \cite[Lemma 4.5]{GMP} that we present for the sake of completeness.  
	\begin{lemma}\label{lemma dis al bordo p}
		Assume $m>0$ and let $0\le f \in L^{N,\infty}(\Omega)$. Let  $u$ be the function defined in Corollary \ref{cor_ex} and let $z$ be the vector field found in Lemma \ref{l.weaksol}. Let  $q>0$, then 
		\begin{equation}\label{dis per il bordo}
			\left|\frac{T^\infty_a(u)^{m(q+1)}}{q+1} - \frac{a^{m(q+1)}}{q+1}\right| \le \left(\frac{a^{mq}}{q} - \frac{T^\infty_a(u)^{mq}}{q}\right) \left[z, \nu\right]\quad \mathcal{H}^{N-1} \text{-a.e  in $\partial\Omega$}, 
		\end{equation}
		for almost every $a>0$.
		In particular $[z,\nu]\leq 0$. 
	\end{lemma}\bk
	\begin{proof}
		Let $u_\varepsilon$ be the solutions of \eqref{e.weaksol} and let us take $\left(\frac{T^\infty_a(u_\varepsilon)^{mq}}{q} - \frac{a^{mq}}{q}\right)\varphi$ with $0\le\varphi \in C^{1}(\overline{\Omega})$ as  test function in \eqref{e.weaksol}. We obtain
		\begin{align}\label{dis per bordo con e}
			\nonumber&\int_{\Omega} u_\varepsilon^m w_\varepsilon \cdot \nabla \left(\frac{T^\infty_a(u_\varepsilon)^{mq}}{q} - \frac{a^{mq}}{q}\right) \varphi  \\  &+ \int_{\Omega} u_\varepsilon^m w_\varepsilon \cdot \nabla \varphi \left(\frac{T^\infty_a(u_\varepsilon)^{mq}}{q} - \frac{a^{mq}}{q}\right) + \alpha_\varepsilon = \int_{\Omega} f \left(\frac{T^\infty_a(u_\varepsilon)^{mq}}{q} - \frac{a^{mq}}{q}\right) \varphi,
		\end{align}
		where $$\alpha_\varepsilon:=\varepsilon \int_{\Omega} \nabla u_\varepsilon \cdot \nabla \varphi \left(\frac{T^\infty_a(u_\varepsilon)^{mq}}{q} - \frac{a^{mq}}{q}\right)+m\varepsilon\int_{\Omega} T^\infty_a(u_\varepsilon)^{mq-1} |\nabla T^\infty_a(u_\varepsilon)|^2 \varphi .$$
		We want to take limit as $\varepsilon$ tends to $0$ in \eqref{dis per bordo con e}. First, we recall that $u_\varepsilon^m w_\varepsilon \rightharpoonup u^m w=z$ *-weakly in $L^\infty(\Omega)^N$ and $\left(\frac{T^\infty_a(u_\varepsilon)^{mq}}{q} - \frac{a^{mq}}{q}\right) \to \left(\frac{T^\infty_a(u)^{mq}}{q} - \frac{a^{mq}}{q}\right)$ strongly in $L^1(\Omega)$.
		Thus,
		\begin{equation}\label{second int}
			\lim_{\varepsilon \to 0} \int_{\Omega} u_\varepsilon^m w_\varepsilon \cdot \nabla \varphi \left(\frac{T^\infty_a(u_\varepsilon)^{mq}}{q} - \frac{a^{mq}}{q}\right) = \int_\Omega z \cdot \nabla \varphi \left(\frac{T^\infty_a(u)^{mq}}{q} - \frac{a^{mq}}{q}\right),
		\end{equation}
		and
		\begin{equation}\label{terzo int}
			\lim_{\varepsilon \to 0} \int_\Omega f \left(\frac{T^\infty_a(u_\varepsilon)^{mq}}{q} - \frac{a^{mq}}{q}\right) \varphi = \int_\Omega f \left(\frac{T^\infty_a(u)^{mq}}{q} - \frac{a^{mq}}{q}\right) \varphi.
		\end{equation}
		We highlight that,  reasoning as  for \eqref{e.limzero},  the first integral of $\alpha_\varepsilon$ tends to $0$   (recall that  $u_\varepsilon$ is uniformly bounded as proved in Lemma \ref{lemma bound}); moreover,   the second integral is nonnegative.
		By \eqref{second int}, \eqref{terzo int},   and by   letting $\varepsilon\to0$ in \eqref{dis per bordo con e}, we   get  
		\begin{align}\label{primo e}
			\nonumber&\limsup_{\varepsilon\to0}\int_{\Omega} u_\varepsilon^m w_\varepsilon \cdot \nabla \left(\frac{T^\infty_a(u_\varepsilon)^{mq}}{q} - \frac{a^{mq}}{q}\right) \varphi \\ &+\int_\Omega z \cdot \nabla \varphi \left(\frac{T^\infty_a(u)^{mq}}{q} - \frac{a^{mq}}{q}\right)\le\int_\Omega f \left(\frac{T^\infty_a(u)^{mq}}{q} - \frac{a^{mq}}{q}\right) \varphi.
		\end{align}
		By virtue of Remark \ref{ext ft}, we gain 
		\begin{equation}\label{tezo int 2}
			\begin{aligned}
				&\int_\Omega f \left(\frac{T^\infty_a(u)^{mq}}{q} - \frac{a^{mq}}{q}\right) \varphi \stackrel{\eqref{56}}{=} \int_\Omega \left(z, D \left[\left(\frac{T^\infty_a(u)^{mq}}{q} - \frac{a^{mq}}{q}\right)\varphi\right]\right)\\ & - \int_{\partial\Omega} \left(\frac{T^\infty_a(u)^{mq}}{q} - \frac{a^{mq}}{q}\right) \varphi \left[z, \nu\right] \, \ensuremath d\mathcal{H}^{N-1}  \stackrel{\eqref{uscire prod}}{=}  \int_\Omega z \cdot \nabla \varphi \left(\frac{T^\infty_a(u)^{mq}}{q} - \frac{a^{mq}}{q}\right) \\& + \int_\Omega \left(z,  \frac{DT^\infty_a(u)^{mq}}{q} \right)\varphi-\int_{\partial\Omega} \left(\frac{T^\infty_a(u)^{mq}}{q}  - \frac{a^{mq}}{q}\right) \varphi \left[z, \nu\right]\, \ensuremath d\mathcal{H}^{N-1}.
			\end{aligned}
		\end{equation}
		Substituting \eqref{tezo int 2} in \eqref{primo e}, we get 
		\begin{equation}\label{ug con lim}
			\begin{aligned}
				&\limsup_{\varepsilon \to 0} \int_\Omega u_\varepsilon^m w_\varepsilon \cdot \nabla \left(\frac{T^\infty_a(u_\varepsilon)^{mq}}{q} - \frac{a^{mq}}{q}\right) \varphi 
				\\
				&\leq \bk\int_\Omega \left(z, D \frac{T^\infty_a(u)^{mq}}{q} \right) \varphi  +\int_{\partial\Omega} \left(\frac{a^{mq}}{q} - \frac{T^\infty_a(u)^{mq}}{q}  \right)   \varphi \left[z, \nu\right]\, \ensuremath d\mathcal{H}^{N-1}.
			\end{aligned}
		\end{equation}
		It only remains to estimate the limit on the left-hand of the previous inequality. 
		Since 
		\begin{equation}\label{primo int}
			\begin{aligned}
				&\int_\Omega u_\varepsilon^m w_\varepsilon \cdot \nabla \left(\frac{T^\infty_a(u_\varepsilon)^{mq}}{q} - \frac{a^{mq}}{q}\right) \varphi =m \int_\Omega u_\varepsilon^{m}\frac{|\nabla u_\varepsilon|^2}{|\nabla u_\varepsilon|_\varepsilon} (T^{\infty }_a)'(u_\varepsilon)T^\infty_a(u_\varepsilon)^{mq-1}\varphi \\ \stackrel{\eqref{dis norm}}{\ge} &\omega_\varepsilon + m\int_\Omega u_\varepsilon^m |\nabla u_\varepsilon| (T^{\infty}_a)'(u_\varepsilon) T^\infty_a(u_\varepsilon)^{mq-1} \varphi = \omega_\varepsilon + \int_\Omega \left|\nabla\left( \frac{T^\infty_a(u_\varepsilon)^{m(q+1)}}{q+1} - \frac{a^{m(q+1)}}{q+1}\right) \right|\varphi,
			\end{aligned}
		\end{equation}
		where
		$$\omega_\varepsilon=-\varepsilon \, m \int_\Omega u_\varepsilon^m (T^\infty_a)'(u_\varepsilon) T^\infty_a(u_\varepsilon)^{mq-1} \varphi .$$
		We note that $\omega_\varepsilon$ is vanishing as $\varepsilon$ tends to $0$ because of \eqref{e.convmL1} and \eqref{e.Linftybound}. Using the weak lower semicontinuity (recall \eqref{sci BV}) in \eqref{primo int},  and   \eqref{ug con lim}, it follows that
		\begin{equation}\label{semicontinuita}
			\begin{aligned}
				&\int_\Omega \left|D\left( \frac{T^\infty_a(u)^{m(q+1)}}{q+1} - \frac{a^{m(q+1)}}{q+1}\right) \right|\varphi + \int_{\partial\Omega} \left|\frac{T^\infty_a(u)^{m(q+1)}}{q+1} - \frac{a^{m(q+1)}}{q+1}\right| \varphi \, d\mathcal{H}^{N-1}\\ \le &\int_\Omega \left(z, D \frac{T^\infty_a(u)^{mq}}{q} \right) \varphi+\int_{\partial\Omega}\left(\frac{a^{mq}}{q} - \frac{T^\infty_a(u)^{mq}}{q}  \right) \varphi \left[z, \nu\right]\, d\mathcal{H}^{N-1}.
			\end{aligned}
		\end{equation}
		Now observe that
		\begin{equation}\label{felicita}
			\begin{aligned}
				\left(z,D\frac{T_{a}^{\infty}(u)^{mq}}{q}\right)
				&\stackrel{\eqref{Tu z parring eq}}{=}\left(z\chi_{\{u>a\}},D\frac{T_{a}^{\infty}(u)^{mq}}{q}\right)\\
				&\stackrel{\eqref{uscire sx}}{=}T_{a}^{\infty}(u)^{m}\left(w\chi_{\{u>a\}},D\frac{T_{a}^{\infty}(u)^{mq}}{q}\right)\\
				&=T_{a}^{\infty}(u)^{m}\theta\left(w\chi_{\{u>a\}},D\frac{T_{a}^{\infty}(u)^{mq}}{q},x\right)\left|D\frac{T_{a}^{\infty}(u)^{mq}}{q}\right|\\
				&\stackrel{\eqref{chain rule pairing}}{=}T_{a}^{\infty}(u)^{m}\theta\left(w\chi_{\{u>a\}},DT_{a}^{\infty}(u),x\right)\left|D\frac{T_{a}^{\infty}(u)^{mq}}{q}\right|\\
				&=\left|D\left(\frac{T_{a}^{\infty}(u)^{m(q+1)}}{q+1}-\frac{a^{m(q+1)}}{q+1}\right)\right|\quad\text{ in $\mathcal{D}'(\Omega)$},
			\end{aligned}
		\end{equation}
		where in the last line we used the fact that $\theta\left(w\chi_{\{u>a\}},DT_{a}^{\infty}(u),x\right)=1\ \ $ $\left|D\frac{T_{a}^{\infty}(u)^{mq}}{q}\right|$-a.e. in $\Omega$ (see   Remark \ref{c.corototvarr}  and observe  that $\left|D\frac{T_{a}^{\infty}(u)^{mq}}{q}\right|\ll |D T_{a}^{\infty}(u)|$).
		
		Substituting \eqref{felicita} in \eqref{semicontinuita}, and using the arbitrariness of $\varphi$, we have proved \eqref{dis per il bordo}. 
	\end{proof}
	\medskip 
	
	\begin{lemma}\label{l.boundary}
		Assume $m>0$ and let $0\le f \in L^{N,\infty}(\Omega)$. Let  $u$ be the function defined in Corollary \ref{cor_ex} and let $z$ be the vector field found in Lemma \ref{l.weaksol}.  Then it holds
		$$\label{sol 3blem}
		\left[z, \nu\right]=-(u^{\Omega})^m \quad \text{$\mathcal{H}^{N-1}$-a.e. on $\partial \Omega \cap \{u>0\}$.}
		$$
	\end{lemma} 
	\begin{proof}
		It follows from Lemma \ref{lemma dis al bordo p}  (recall that $[z,\nu]\leq 0$) that, for almost every $a>0$ and for all $q>0$, it holds
		\begin{align*}
			\frac{q}{q+1}&\left(T^\infty_a(u)^{m(q+1)}-a^{m(q+1)}\right) \\ & \nonumber\le \left(a^{mq}-T^\infty_a(u)^{mq}\right)[z, \nu]   = \left(T^\infty_a(u)^m a^{mq}-T^\infty_a(u)^{m(q+1)}\right)\frac{[z,\nu]}{T^\infty_a(u)^m} \\ &\stackrel{T^\infty_a(u) \ge a}{\le} \left(a^{m(q+1)}- T^\infty_a(u)^{m(q+1)}\right) \frac{[z, \nu]}{T^\infty_a(u)^m} \quad \text{$\mathcal{H}^{N-1}$-a.e. on $\partial \Omega$.}
		\end{align*}
		As a consequence, we gain 
		\begin{equation*}
			\frac{q}{q+1} \le - \frac{[z,\nu]}{T^\infty_a(u)^m}  \quad \text{$\mathcal{H}^{N-1}$-a.e. on $\partial\Omega \cap \{u>0\}$.}
		\end{equation*}
		Taking limits as  $a$ tends to $0$  and $q$  to $\infty$ in the previous inequality, it follows that
		\begin{equation}\label{mezza diz}
			(u^{\Omega})^m\bk \le -[z,\nu] \quad \text{$\mathcal{H}^{N-1}$-a.e. on $\partial \Omega \cap \{u>0\}$.}
		\end{equation}
		We now show the reverse inequality. We know that $\chi_{\{u>a\}}\in BV(\Omega)$ and $w\chi_{\{u>a\}}\in {\DM} (\Omega)$ for almost every  $a>0$. As a consequence, \eqref{= al bordo mis} and the fact that $z\chi_{\{u>a\}}=T_{a}^{\infty}(u)^{m}w\chi_{\{u>a\}}$ in $\Omega$ with $\|w\|_{L^{\infty}(\Omega)^N}\leq1$ imply that 
		\begin{equation}\label{equ brack uno}
			|[z,\nu]\chi_{\{u>a\}}|\le T^\infty_a(u)^m \quad \text{$\mathcal{H}^{N-1}$-a.e. on $\partial \Omega$ and for almost all }a>0.
		\end{equation}
		Letting $a\to0$ in \eqref{equ brack uno} and using the fact that $\chi_{\{u>a\}}=1$  in $\{u>0\}\cap\partial\Omega$  for $a$ sufficiently small, we get
		$$
		|[z,\nu]|\le (u^{\Omega})^m\quad \mathcal{H}^{N-1}\text{-a.e on }\{u>0\}\cap \partial\Omega.
		$$
		In view of \eqref{mezza diz}, we conclude that 
		$$
		[z,\nu]=- (u^{\Omega})^{m}\bk\text{ in }\partial \Omega \cap \{u>0\},
		$$
		which is the desired result.
	\end{proof}
	We finish by proving Theorem \ref{teo f N}.
	\begin{proof}[Proof of Theorem \ref{teo f N}]
		As we said, the proof of Theorem \ref{teo f N}  is a consequence   of the previous results.
		From Corollary \ref{cor_ex} and Lemma \ref{u reg} one deduces the existence of a limit function $u \in DTBV^+(\Omega) \cap L^\infty(\Omega)$. Additionally, from Lemma \ref{l.weaksol}, one gets the existence of the limit  vector field $w \in L^\infty(\Omega)^N$ with $\|w\|_{L^\infty(\Omega)^N} \le 1$. The quantity $z:=u^m w \in \DM(\Omega)$ satisfies \eqref{sol 1} as shown in Lemma \ref{l.weaksol}. Lemma \ref{c.corototvar} gives the validity of \eqref{sol 2 b}.
		Lemma \ref{l.boundary} proves that the boundary condition \eqref{sol 3b} holds. This concludes the proof.
	\end{proof}

	\section{The problem with a sign-changing $f$}\label{s.signchanging}
	In this section we assume that $f \in L^{N, \infty}(\Omega)$ with a generic changing sign.  The proofs we exhibit are technical re-adaptions of the ideas of the previous sections. Hence,   here, we mainly  focus on the difficulties  arising from the no sign assumption on $f$. \bk

	For $m>0$ let us consider the following   problem \begin{equation}\label{prob segno}
		\begin{cases}
			- \dis \operatorname{div}\left(|u|^m \frac{Du}{|Du|}\right)=f & \text{in $\Omega$,}\\
			u=0 & \text{on $\partial \Omega$.}
		\end{cases}
	\end{equation} 
	
	\medskip 
	Let us determine
	how a solution should be intended in this case. 
	
	\begin{defin}\label{def sol changing}
		Assume $m>0$ and let  $f \in L^{N,\infty}(\Omega)$. A function $u \in DTBV(\Omega)\cap L^\infty(\Omega)$ is a solution to \eqref{prob segno}, if there exists a vector field $ w \in L^\infty(\Omega)^N$ such that $\|w\|_{L^\infty(\Omega)^N} \le 1$, such that  the vector field $z:=|u|^m w \in \DM(\Omega)$ satisfies 
		\begin{equation}\label{sol 1 segno}
			-\operatorname{div}z=f \quad \text{as measures in $\Omega$,}
		\end{equation}
		\begin{equation}\label{sol 2 +}
			\left(z, DT^\infty_a(u)\right) = \frac{1}{m+1}|DT^\infty_a(u)^{m+1}| \quad \text{in $\mathcal{D}'(\Omega)$, for a.e. $a>0$,}
		\end{equation}
		\begin{equation}\label{sol 2 -}
			\left(z, DT_{-\infty}^{-a}(u)\right) = \frac{1}{m+1}|D|T_{-\infty}^{-a}(u)|^{m+1}| \quad \text{in $\mathcal{D}'(\Omega)$, for a.e. $a>0$,}
		\end{equation}
		\begin{equation}\label{sol 3 +}
			\left[z, \nu\right]=-((u^{+})^{\Omega})^m \quad \text{$\mathcal{H}^{N-1}$-a.e. on $\partial \Omega \cap \{u>0\}$,} 
		\end{equation}
		and \begin{equation}\label{sol 3 -}
			\left[z, \nu\right]=((u^{-})^{\Omega})^m \quad \text{$\mathcal{H}^{N-1}$-a.e. on $\partial \Omega \cap \{u<0\}$}.
		\end{equation}
	\end{defin}
	\begin{remark} 
		We underline the main difference with the case of a nonnegative $f \in L^{N,\infty}(\Omega)$. First of all we note that, as $u$ changes sign in general, then $z:=|u|^m w$. Moreover \eqref{sol 1 segno}, \eqref{sol 2 +} and \eqref{sol 2 -} explain the role of the vector field $z$. Finally \eqref{sol 3 +} and \eqref{sol 3 -} clarify how the datum is assumed on the boundary. Observe that a nonnegative  function  $u$ is a solution in the sense  of Definition \ref{def sol changing} if and only if it is a solution in the sense of Definition \ref{def sol}. 
		
		\medskip 
		The main result of this section is the following. 
	\end{remark}
	\begin{theorem}\label{teo f sgn}
		Assume $m>0$ and let  $f \in L^{N,\infty}(\Omega)$. Then there exists a solution $u$ of problem \eqref{prob segno} in the sense of Definition \ref{def sol changing}. In particular,  if $f \not \equiv 0$, then $u \not \equiv 0$.
	\end{theorem}

	\medskip

	As before, we split the proof of Theorem \ref{teo f sgn} into different lemmas.   As the proofs are technical adjustments of the proofs of the corresponding  results proven in Section \ref{s.proof}   we try to sketch the proof by  highlighting  the main differences   with the case of  nonnegative datum $f \in L^{N,\infty}(\Omega)$. Again we reason by approximating with solutions $u_\varepsilon$ of problem \eqref{e.problemapprox} given by Lemma \ref{l.lemmaapprox}.

	\medskip 
	
	\begin{lemma}\label{l.lemmaapprox sgn}
		Assume $m>0$ and let $f \in L^{\widetilde{m}}(\Omega)$ with $\widetilde{m}$ defined in \eqref{mtilde}. 
		Let $u_{\varepsilon} $ be the   solution of problem \eqref{e.problemapprox} and  let   $C_{\varepsilon}$ be the constant defined by \eqref{ceps}. 
		
		Then it holds: there exists $\overline{\varepsilon}$ such that
		$$
		\||u_{\varepsilon}|^{m+1}\|_{L^{1^{*}}(\Omega)}\leq C_{\varepsilon}^{m+1}\text{ for all }0<\varepsilon<\overline{\varepsilon},
		$$
		and
		\begin{equation}\label{stime approx BV sgn}
			\||u_{\varepsilon}|^{m+1}\|_{BV(\Omega)}\leq (m+1)\left(\varepsilon|\Omega|^{1-\frac{m}{(m+1)1^{*}}}C_{\varepsilon}^{m}+\|f\|_{L^{\widetilde{m}}(\Omega)}C_{\varepsilon}\right)\text{ for all }0<\varepsilon<\overline{\varepsilon}.
		\end{equation}
		In particular, the sequence $u_\varepsilon^{m+1} $ is uniformly bounded in $BV(\Omega)$ for any $0<\varepsilon<\overline{\varepsilon}$.
		
		Finally, if $f\in L^{N,\infty}(\Omega)$, then  the sequence $u_{\varepsilon}$ is uniformly bounded in $L^{\infty}(\Omega)$, with 
		\begin{equation}\label{e.Linftybound sgn}
			\|u_{\varepsilon}\|_{L^{\infty}(\Omega)}\leq k_{0,\tau}+\left(\frac{\mathcal{{S}}_{1}\varepsilon}{\tau}\right) 2^N|\Omega|^{\frac{1}{N}}\text{ for all }0<\varepsilon, \tau<1,
		\end{equation}
		where
		$$
		k_{0,\tau}=\left(\frac{{\mathcal{\widetilde{S}}_{1}\|f\|_{L^{N,\infty}(\Omega)}}}{1-\tau}\right)^{\frac{1}{m}}.
		$$
	\end{lemma}
	\begin{proof}
		Estimates \eqref{stime approx BV sgn} and \eqref{e.Linftybound sgn} follow as in the proof of Lemmas  \ref{l.stimeLq} and \ref{lemma bound}. 
	\end{proof}
	We now identify the  almost everywhere limit $u$ of the sequence $u_{\varepsilon}$.
	\begin{corollary}\label{lemma ae conv}
		Assume $m>0$, let $f \in L^{N,\infty}(\Omega)$ and let $u_{\varepsilon}$ be a sequence of solutions of problem \eqref{e.problemapprox}. There exists a function $u\in L^{\infty}(\Omega)$ such that $$\label{e.convmL1 sgn}
		u_{\varepsilon}\to u  \quad \text{strongly in $L^r(\Omega)$, for all $1 \le r < \infty$,  almost everywhere in $\Omega$}
		$$ and
		\begin{equation}\label{e.boundexplLinfty sgn}
			\|u\|_{L^{\infty}(\Omega)}\leq(\mathcal{\widetilde{S}}_{1}\|f\|_{L^{N,\infty}(\Omega)})^{\frac{1}{m}}.	
		\end{equation}	
		Moreover, $|u|^{m+1}\in BV(\Omega)$ with
		\begin{equation}\label{stime BV sgn}
			\||u|^{m+1}\|_{BV(\Omega)}\leq\mathcal{S}_{1}^{\frac{1}{m}}((m+1)\|f\|_{L^{\widetilde{m}}(\Omega)})^{\frac{m+1}{m}}.
		\end{equation}
	\end{corollary} 
	\begin{proof}
		We first establish almost everywhere convergence of $u_{\varepsilon}$. Let $a>0$ and take $\varphi= T_{a}^{\infty}(u_{\varepsilon})-a$ as a test function in \eqref{e.weaksol}. We get
		$$
		\int_{\{u_{\varepsilon}\geq a\}}|u_{\varepsilon}|^{m}\frac{|\nabla u_{\varepsilon}|^2}{|\nabla u_{\varepsilon}|_{\varepsilon}}\leq\int_{\Omega}f(T_{a}^{\infty}(u_{\varepsilon})-a).
		$$
		H\"older's inequality and Lemma \ref{l.lemmaapprox sgn} imply that the right-hand of the previous is bounded by a constant $C>0$. To handle  the left-hand, we use \eqref{dis norm} and the fact that the integral is on $\{u_{\varepsilon}\geq a\}$. Therefore one has
		$$
		\int_{\{u_{\varepsilon}\geq a\}}|\nabla u_{\varepsilon}|\leq\varepsilon|\Omega|+\frac{C}{a^{m}}.
		$$ 
		Consequently, for each $a>0$ the sequence $|\nabla T_{a}^{\infty}(u_{\varepsilon})|$ is uniformly bounded in   $L^1(\Omega)$ with respect to $\varepsilon$. On the other hand, Lemma \ref{l.lemmaapprox sgn} implies that $T_{a}^{\infty}(u_{\varepsilon})$ is uniformly bounded in $L^{\infty}(\Omega)$ (and so in $L^{1}(\Omega)$). Thus $T_{a}^{\infty}(u_{\varepsilon})$ is bounded in $BV(\Omega)$. The compactness of the embedding $BV(\Omega)\hookrightarrow L^{1}(\Omega)$ implies that there exists $v_{a}\in BV(\Omega)$ such that, up to a subsequence,
		\begin{equation}\label{equ zero}
			T_{a}^{\infty}(u_{\varepsilon})\to v_{a}\text{ in }L^{1}(\Omega)\quad\text{ and }\quad T_{a}^{\infty}(u_{\varepsilon})\to v_{a}\text{ a.e in }\Omega\text{ for all }a>0.
		\end{equation}
		So far we have considered the ``positive part'' of $u_{\varepsilon}$. To consider the negative one, we take as test function $\varphi= T_{-\infty}^{-a}(u_{\varepsilon})+a$ as test. Similarly, we get
		$$
		\int_{\{u_{\varepsilon}\leq -a\}}|u_{\varepsilon}|^{m}\frac{|\nabla u_{\varepsilon}|^2}{|\nabla u_{\varepsilon}|_{\varepsilon}}\leq\int_{\Omega}f|T_{-\infty}^{-a}(u_{\varepsilon})+a|.
		$$
		and thus, as before
		$$
		\int_{\{u_{\varepsilon}\leq -a\}}|\nabla u_{\varepsilon}|\leq\varepsilon|\Omega|+\frac{C}{a^{m}}.
		$$ 
		We conclude that the sequence $T_{-\infty}^{-a}(u_{\varepsilon})$ is bounded in $BV(\Omega)$ and thus that there exists $v_{-a}\in BV(\Omega)$ such that, up to a subsequence,
		\begin{equation}\label{equ uno}
			T_{-\infty}^{-a}(u_{\varepsilon})\to v_{-a}\text{ in }L^{1}(\Omega)\quad\text{ and }\quad T_{-\infty}^{-a}(u_{\varepsilon})\to v_{-a}\text{ a.e in }\Omega\text{ for all }a>0.
		\end{equation}
		We may use a diagonal argument to obtain a set $\widetilde{\Omega}\subset\Omega$ with $|\Omega\setminus\widetilde{\Omega}|=0$ and a sequence $\varepsilon_{j}\to0$  such that
		$$
		T_{a}^{\infty}(u_{\varepsilon_j}(x))\to v_{a}(x)\text{ for all }x\in\widetilde{\Omega}\text{ and }a>0,
		$$
		and
		$$
		T^{-a}_{-\infty}(u_{\varepsilon_j}(x))\to v_{-a}(x)\text{ for all }x\in\widetilde{\Omega}\text{ and }a>0.
		$$
		Now we define the following sets
		$$
		E=\{x\in\widetilde{\Omega}:\text{ there exists }a>0\text{ such that }v_{a}(x)>a\},
		$$
		and
		$$
		F=\{x\in\widetilde\Omega:\text{ there exists }a>0\text{ such that }v_{-a}(x)<-a\}.
		$$
		It is clear that $E\cap F=\emptyset$. Furthermore observe that if $v_{a_{0}}(x)>a_0$ for some $a_0>0$, then the uniqueness of the limit implies that $v_{a}(x)=v_{a_0}(x)$ for all $0<a\leq a_0$. On the other hand, the sequence $v_{a}$ is non-decreasing in $a$. Consequently, the function
		$$
		\overline{u}(x):=\lim_{a\to0}v_{a}(x), 
		$$
		is well defined for all $x\in\widetilde\Omega$.

		We claim that
		$
		u_{\varepsilon}\to \overline{u}\text{ pointwise in }E.
		$
		Indeed, for all $x\in E$, there exist $a_0>0$ and $\varepsilon_{x}>0$ such that $T_{a_0}^{\infty}(u_{\varepsilon}(x))>a_0$ for all $0<\varepsilon<\varepsilon_x$. Consequently, $T_{a_0}^{\infty}(u_{\varepsilon}(x))=u_{\varepsilon}(x)$ for all $0<\varepsilon<\varepsilon_{x}$. The claim then follows from \eqref{equ zero} and from the fact that $v_{a_0}(x)=\overline{u}(x)$. We now argue similarly for the set $F$. First of all, observe that if $v_{-a_{0}}(x)<-a_0$ for some $a_0>0$, then the uniqueness of the limit implies that $v_{-a}(x)=v_{-a_0}(x)$ for all $0<a\leq a_0$. On the other hand, the sequence $v_{-a}$ is non-decreasing in $a$. Consequently, the function
		$$
		\underline{u}(x):=\lim_{a\to0}v_{-a}(x),
		$$
		is well defined for all $x\in\widetilde\Omega$.

		We claim that
		$
		u_{\varepsilon}\to \underline{u}\text{ pointwise in }F.
		$
		Indeed, for all $x\in F$, there exist $a_0>0$ and $\varepsilon_{x}>0$ such that $T^{-a_0}_{-\infty}(u_{\varepsilon}(x))<-a_0$ for all $0<\varepsilon<\varepsilon_x$. Consequently, $T^{-a_0}_{-\infty}(u_{\varepsilon}(x))=u_{\varepsilon}(x)$ for all $0<\varepsilon<\varepsilon_{x}$. The claim then follows from \eqref{equ uno} and from the fact that $v_{-a_0}(x)=\underline{u}(x)$. We finish by showing that
		$$
		u_{\varepsilon}\to0\text{ a.e in }\widetilde\Omega\setminus (E\cup F).
		$$
		Indeed, assume by contradiction that $u$ does not converge to $0$ a.e in $\widetilde\Omega\setminus (E\cup F)$. Then, there exists a set $V\subset \widetilde\Omega\setminus (E\cup F)$ with positive measure such that $u_{\varepsilon}(x)$ does not converge to zero for all $x\in V$. That is, for each $x\in V$ there exists $\varepsilon_{0,x}$ and a sequence $\varepsilon_j\to0$ such that $|u_{\varepsilon_{j}}(x)|>\varepsilon_{0,x}$ for all $j\in\mathbb{N}$. But this implies that $x\in E\cup F$, which is a contradiction. We have thus shown that $u_{\varepsilon}$ converges a.e in $\Omega$ to the function defined by
		$$
		u:=\begin{cases}
			\overline{u}\text{ in } E,\\
			\underline{u}\text{ in }F,\\
			0\text { in }\Omega\setminus{E\cup F}.	
		\end{cases}	
		$$
		Estimates \eqref{e.boundexplLinfty sgn} and \eqref{stime BV sgn} are consequences of Lemma \ref{l.lemmaapprox sgn}. This proves the result.
	\end{proof}
	Now we show the existence of the limit  vector field $w$ (and so $z$).
	\begin{lemma}\label{l.weaksol sgn}
		Assume $m>0$, let $f \in L^{N,\infty}(\Omega)$ and let $u$ be the function given by Corollary \ref{lemma ae conv}. There exists a vector field $w \in L^\infty(\Omega)^N$ with $\|w\|_{L^\infty(\Omega)^N} \leq 1$,  such that $z:=|u|^m w\in \DM(\Omega)$ satisfies  \eqref{sol 1 segno}. Moreover if $f \not \equiv 0$, then $u \not \equiv 0$.
	\end{lemma}
	\begin{proof} 
		The proof is the same of Lemma \ref{l.weaksol}.
	\end{proof}
	The following lemmas highlight the meaning of the vector field $z$ and show that $u$ has no jump part.
	\begin{lemma}\label{l.totvar sgn}
		Assume $m>0$ and let $f \in L^{N,\infty}(\Omega)$. Let $u$ be the function given by Corollary \ref{lemma ae conv} and $z$ be the vector field defined in Lemma \ref{l.weaksol sgn}. We have 
		\begin{equation}\label{e.totvar +}
			\left(z,DT_{a}^{\infty}(u)\right)\geq \frac{1}{m+1}|DT_{a}^{\infty}(u)^{m+1}| \quad \text{as measures, for all $a>0$,}	
		\end{equation}
		and
		\begin{equation}\label{e.totvar -}
			\left(z,DT^{-a}_{-\infty}(u)\right)\geq \frac{1}{m+1}|D|T^{-a}_{-\infty}(u)|^{m+1}| \quad \text{as measures, for all $a>0$.}	
		\end{equation}
	\end{lemma}
	\begin{proof} 
		The proof of \eqref{e.totvar +} is the same of \eqref{e.totvar}. Let us show \eqref{e.totvar -}. We take $\varphi\in C_{c}^{1}(\Omega)$ with $\varphi\geq0$ and we choose $T_{-\infty}^{-a}(u_{\varepsilon})\varphi$ as a test function in \eqref{e.weaksol}. We get
		$$
		\int_{\Omega} |u_{\varepsilon}|^{m}w_{\epsilon}\cdot\nabla T_{-\infty}^{-a}(u_{\varepsilon})\,\varphi+\int_{\Omega}|u_{\varepsilon}|^{m}w_{\varepsilon}\cdot\nabla \varphi \,T_{-\infty}^{-a}(u_{\varepsilon})+\varepsilon\int_{\Omega}\nabla u_{\varepsilon}\cdot\nabla (T_{-\infty}^{-a}(u_{\varepsilon})\varphi)=\int_{\Omega}f T_{-\infty}^{-a}(u_{\varepsilon})\varphi.
		$$
		Arguing as in the proof of Lemma \ref{l.totvar}, we get
		$$
		\frac{1}{m+1}\int_{\Omega}|D |T_{-\infty}^{-a}(u)|^{m+1}|\varphi+\int_{\Omega}z\cdot\nabla\varphi\, T_{-\infty}^{-a}(u)\leq\int_{\Omega}fT_{-\infty}^{-a}(u)\varphi.
		$$
		Through the definition of pairing, using \eqref{sol 1 segno} (recall \eqref{def pair div l1}), we gain
		$$\frac{1}{m+1}\int_{\Omega}|D |T^{-a}_{-\infty}(u)|^{m+1}| \varphi \le \int_{\Omega}(z, DT^{-a}_{-\infty}(u)) \varphi.$$
		This proves the result.
	\end{proof}
	We want to conclude that equality holds in \eqref{e.totvar +} and \eqref{e.totvar -}. To do that, we first prove that $u$ does not jump.
	\begin{lemma}\label{u reg sgn}
		Assume $m>0$ and let $f \in L^{N,\infty}(\Omega)$. Let $u$ be the function given by Corollary \ref{lemma ae conv}. Then $u \in DTBV(\Omega) \cap L^\infty(\Omega)$.
	\end{lemma}
	\begin{proof}
		The proof that $u^+ \in DTBV^+(\Omega)$ is as in Lemma \ref{u reg}. Let us  show that  $u^- \in DTBV^+(\Omega)$. As by Lemma \ref{l.weaksol sgn}  $z\in \DM(\Omega)$,  we first  claim  that $w\chi_{\{u<-a\}}\in \DM(\Omega)$ for a.e $a>0$; this follows by Lemma \ref{campo prodotto} as 
		$$
		w\chi_{\{u<-a\}}=|u|^{-m}z\chi_{\{u<-a\}},
		$$
		and $|u|^{-m}\chi_{\{u<-a\}}\in BV(\Omega)\cap L^{\infty}(\Omega)$, since $u\in TBV(\Omega)\cap L^{\infty}(\Omega)$. 
		In Lemma \ref{l.weaksol sgn} we showed that  equation \eqref{sol 1 segno} holds. This, in turn, implies that $\operatorname{div}z \in L^{N,\infty}(\Omega)$.
		Moreover, by \cite[Proposition 3.69]{AFP} we know that $J_{T^{-a}_{-\infty}(u)}=J_{|T^{-a}_{-\infty}(u)|^{m+1}}$ and $\nu_{T^{-a}_{-\infty}(u)}=\nu_{|T^{-a}_{-\infty}(u)|^m}$ on $J_{T^{-a}_{-\infty}(u)}$ for a.e. $a>0$ since $m>0$.
		\\Thanks to Lemma \ref{lem sigma} and the fact that  $\operatorname{div}z \in L^{N,\infty}(\Omega)$, we may repeat the computations in the proof of Lemma \ref{u reg} and conclude that
		$$0=\mathcal{H}^{N-1}\left(J_{T^{-a}_{-\infty}(u)}\right)=\mathcal{H}^{N-1}\left(S_{T^{-a}_{-\infty}(u)}\right) \quad \text{for a.e. $a>0$.}$$
		This concludes the proof.
	\end{proof}
	Now we show that inequalities  \eqref{e.totvar +} and \eqref{e.totvar -}  obtained in Lemma \ref{l.totvar sgn} are actually  equalities.
	\begin{corollary}\label{c.corototvar sgn}
		Assume $m>0$ and let $f \in L^{N,\infty}(\Omega)$. Let $u$ be the function given by Corollary \ref{lemma ae conv} and $z$ be the vector field defined in Lemma \ref{l.weaksol sgn}. Then   \eqref{sol 2 +} and \eqref{sol 2 -} hold.
	\end{corollary}
	\begin{proof} It is a straightforward application of \cite[Lemma 5.10]{GMP}, but for the sake of completeness, we present the details. The proof of \eqref{sol 2 +} is almost identical to the one of Lemma \ref{c.corototvar}.
		
		\smallskip 
		We focus on  \eqref{sol 2 -}.  
		It holds  
		\begin{align*}
			\frac{1}{m+1}|D|T^{-a}_{-\infty}(u)|^{m+1}| \stackrel{\eqref{e.totvar -}}{\le} & \left(z, DT^{-a}_{-\infty}(u)\right) \stackrel{\eqref{Tu z parring eq}}{=} \left(z \chi_{\{u<-a\}}, DT^{-a}_{-\infty}(u)\right) \\  \stackrel{\eqref{uscire sx}}{=} & |T^{-a}_{-\infty}(u)|^m \left(w \chi_{\{u<-a\}}, DT^{-a}_{-\infty}(u)\right) \\ \stackrel{\eqref{ec:2}}{\leq} & |T^{-a}_{-\infty}(u)|^m   |DT^{-a}_{-\infty}(u)| 
			\\   \stackrel{\eqref{chain rule senza j}}{=}  &   \frac{1}{m+1}|D|T^{-a}_{-\infty}(u)|^{m+1}|, 
		\end{align*} 
		where in  the last equality, we used  that  $u\in DTBV(\Omega)$.
		This proves  \eqref{sol 2 -} and the proof is complete. 
	\end{proof}
	
	\begin{remark} 
		As before, we deduce that \eqref{sol 2 +} and \eqref{sol 2 -} can be equivalently recast as 
		$$
		\left(w \chi_{\{u>a\}}, DT^\infty_a(u)\right)=\left|D T^\infty_a(u)\right|  \quad \text{in $\mathcal{D}'(\Omega)$, for a.e. $a>0$,}
		$$
		and
		$$
		\left(w \chi_{\{u<-a\}}, DT^{-a}_{-\infty}(u)\right)=\left|D |T^{-a}_{-\infty}(u)|\right|  \quad \text{in $\mathcal{D}'(\Omega)$, for a.e. $a>0$,}
		$$
		in Definition \ref{def sol changing}. 
		
	\end{remark}
	
	Now we study the behavior of the solution $u$ on the boundary $\partial \Omega$. We start with the extension of   Lemma \ref{lemma dis al bordo p} to this case. 
	\begin{lemma}\label{lemma dis al bordo con segno}
		Assume $m>0$ and let $f \in L^{N,\infty}(\Omega)$. Let $u$ be the function given by Corollary \ref{lemma ae conv} and $z$ be the vector field defined in Lemma \ref{l.weaksol sgn}. Then, for every $q>0$   \begin{equation}\label{dis bordo +}
			\left|\frac{T^\infty_a(u)^{m(q+1)}}{q+1} - \frac{a^{m(q+1)}}{q+1}\right| \le \left(\frac{a^{mq}}{q} - \frac{T^\infty_a(u)^{mq}}{q}\right) \left[z, \nu\right],\end{equation} \text{for a.e.  $a>0$,\ \text{and $\mathcal{H}^{N-1}$-a.e. on $\partial\Omega \cap \{u>0\}$.}} In particular, $\left[z, \nu\right]\leq 0$ $\mathcal{H}^{N-1}$-a.e. on $\partial\Omega \cap \{u>0\}$, 
		and 
		\begin{equation}\label{dis bordo -}
			\left|\frac{(-T^{-a}_{-\infty}(u))^{m(q+1)}}{q+1} - \frac{a^{m(q+1)}}{q+1}\right| \le \left(\frac{(-T^{-a}_{-\infty}(u))^{mq}}{q} - \frac{a^{mq}}{q}\right) \left[z, \nu\right], 	\end{equation} 
		\text{for a.e.  $a>0$,\ \text{and $\mathcal{H}^{N-1}$-a.e. on $\partial\Omega \cap \{u<0\}$.}} In particular, $\left[z, \nu\right]\geq 0$ $\mathcal{H}^{N-1}$-a.e. on $\partial\Omega \cap \{u<0\}$.
		
	\end{lemma}
	\begin{proof}
		The proof of \eqref{dis bordo +} is the same of \eqref{dis per il bordo}. For \eqref{dis bordo -},	we choose $\left(\frac{a^{mq}}{q} - \frac{(-T^{-a}_{-\infty}(u_\varepsilon))^{mq}}{q}\right)\varphi$ with $0\le\varphi \in C^{1}(\Omega)$ as test function in \eqref{e.weaksol}. The result is then obtained by the same reasoning described in the proof of Lemma \ref{lemma dis al bordo p}.
	\end{proof}
	\begin{lemma}\label{l.boundary con segno}
		Assume $m>0$ and let $f \in L^{N,\infty}(\Omega)$. Let $u$ be the function given by Corollary \ref{lemma ae conv} and $z$ be the vector field defined in Lemma \ref{l.weaksol sgn}. Then, \eqref{sol 3 +} and \eqref{sol 3 -} hold.
	\end{lemma}	
	\begin{proof}
		The proof of \eqref{sol 3 +} is the same of \eqref{sol 3b}. Here  we show \eqref{sol 3 -}.
		\\From \eqref{dis bordo -}, we have that for almost every $a>0$ and for all $q>0$
		\begin{align*}
			\frac{q}{q+1}&\left|(-T^{-a}_{-\infty}(u))^{m(q+1)} - a^{m(q+1)}\right| \le \left((-T^{-a}_{-\infty}(u))^{mq} - a^{mq}\right) \left[z, \nu\right] \\= &\left(\left(-T^{-a}_{-\infty}(u)\right)^{m(q+1)}-\left(-T^{-a}_{-\infty}(u)\right)^m a^{mq} \right)\frac{[z,\nu]}{\left(-T^{-a}_{-\infty}(u)\right)^m} \\ \stackrel{-T^{-a}_{-\infty}(u) \ge a}{\le} &\left(\left(-T^{-a}_{-\infty}(u)\right)^{m(q+1)}-a^{m(q+1)} \right) \frac{[z, \nu]}{\left(-T^{-a}_{-\infty}(u)\right)^m} \quad \text{$\mathcal{H}^{N-1}$-a.e. on $\partial\Omega \cap \{u<0\}$.}
		\end{align*}
		As a consequence, we gain 
		\begin{equation*}
			\frac{q}{q+1} \le \frac{[z,\nu]}{\left(-T^{-a}_{-\infty}(u)\right)^m}  \quad \text{$\mathcal{H}^{N-1}$-a.e. on $\partial\Omega \cap \{u<0\}$.}
		\end{equation*}
		Taking limits as $q$ tends to $\infty$ and $a$ to $0$ in the previous inequality , it follows that 
		\begin{equation}\label{mezza diz -}
			(-u^{\Omega})^m \le [z,\nu] \quad \text{$\mathcal{H}^{N-1}$-a.e. on $\partial \Omega \cap \{u<0\}$.}
		\end{equation}
		On the other hand, similarly to \eqref{equ brack uno}, we know that for almost every $a>0$ 
		$$|[z,\nu]\chi_{\{u<-a\}}|\le \left(-T^{-a}_{-\infty}(u)\right)^m \quad \text{$\mathcal{H}^{N-1}$-a.e. on $\partial\Omega \cap \{u<0\}$.}$$ 
		Letting $a\to 0$ and using the fact that $\chi_{\{u<-a\}}=1$ in $\partial\Omega \cap \{u<0\}$ for $a$ sufficiently small, we get
		$$|[z,\nu]| \le (-u^{\Omega})^m\text{ for all }x\in\partial\Omega \cap \{u<0\}.$$
		Taking this  together with \eqref{mezza diz -}, we conclude that 
		$$
		[z,\nu]=(-u^{\Omega})^{m}\text{ in }\partial \Omega \cap \{u<0\}.
		$$
		This proves the result. \bk
	\end{proof}
	Finally, we prove Theorem \ref{teo f sgn}.
	\begin{proof}[Proof of Theorem \ref{teo f sgn}]
		The proof of Theorem \ref{teo f sgn}  is an immediate  consequence of the previous results.
		\\In Corollary  \ref{lemma ae conv} and Lemma \ref{u reg sgn}, we establish the existence of a function $u \in DTBV(\Omega) \cap L^\infty(\Omega)$. Additionally, in Lemma \ref{l.weaksol sgn}, we demonstrate the existence of a vector field $w \in L^\infty(\Omega)^N$ with $\|w\|_{L^\infty(\Omega)^N} \le 1$ such that $z:=|u|^m w \in \DM(\Omega)$ satisfies \eqref{sol 1 segno}, and in Corollary \ref{c.corototvar sgn}  we proved \eqref{sol 2 +}, \eqref{sol 2 -}.
		\\The boundary condition is satisfied by $u$ in the sense of \eqref{sol 3 +} and \eqref{sol 3 -}, as proven in Lemma \ref{l.boundary con segno}.
		This concludes the proof\end{proof}

	\section{Some explicit  examples and remarks} \label{sec7}
	In this section we  construct some  example of solutions of problem \eqref{prob}. Concerning the first one, we need the following definition. 
	\begin{defin}\label{def calib}
		We say that a bounded convex set $E$ of class $C^{1,1}$ is calibrable if there exists a vector field $\xi\in L^{\infty}(\mathbb{R}^{N})^{N}$ such that $\|\xi\|_{L^{\infty}(\mathbb{R}^{N})^N}\leq1$, $(\xi,D\chi_{E})=|D\chi_{E}|$ as measures and 
		$$
		-\operatorname{div}\xi=\lambda_{E}\chi_{E}\text{ in }\mathcal{D}'(\mathbb{R}^N)
		$$
		for some constant $\lambda_{E}$. In this case, $\lambda_{E}=\frac{Per(E)}{|E|}$ and $[\xi,\nu]=-1\, \mathcal{H}^{N-1}$-a.e in $\partial E$, see \cite{DGOP}.
	\end{defin}
	By \cite[Theorem 9]{acc} a bounded and convex set $E$ is calibrable if and only if 
	$$
	(N-1)\|\mathbf{H}_{E}\|_{L^{\infty}(\partial E)}\leq\lambda_{E}=\frac{Per(E)}{|E|},
	$$
	where $\mathbf{H}_{E}$ denotes the mean curvature of $\partial E$.
	\begin{example}
		Let $\Omega$ be  a calibrable set. We will prove that the function $u=\left(\frac{|\Omega|}{Per(\Omega)}\right)^{\frac1m}$ is a solution of \eqref{prob} in the sense of Definition \ref{def sol}. Indeed, considering the restriction to $\Omega$ of the vector field in Definition \ref{def calib}, that is, $w=\xi|_{\Omega}$, we get
		$$
		-\operatorname{div}\left(\frac{|\Omega|}{Per(\Omega)}w\right)=1\ \ \text{in}\ \Omega.
		$$
		Consequently, the function $u=\left(\frac{|\Omega|}{Per(\Omega)}\right)^{\frac1m}$ solves
		$$
		-\operatorname{div}(u^m w)=1\text{ in }\mathcal{D}'(\Omega)\qquad\text{ and }\quad[u^m w,\nu]=u^m [w,\nu]=-u^m\quad\mathcal{H}^{N-1}-a.e\text{ on }\partial\Omega.
		$$
		Moreover, using \eqref{int per parti} with $z=u^m w$, we get
		\begin{align*}
			\int_{\Omega}(z,DT_{a}^{\infty}(u))&= -\int_{\Omega}T_{a}^{\infty}(u){\rm div} z\bk +\int_{\partial\Omega}T_{a}^{\infty}(u)[z,\nu] \, d\mathcal{H}^{N-1}\\
			&=T_{a}^{\infty}(u)|\Omega|+T_{a}^{\infty}(u)(-u^m)Per(\Omega)\\
			&=T_{a}^{\infty}(u)\left(|\Omega|-u^m Per(\Omega)\right)=0=\frac{1}{m+1}|D (T_{a}^{\infty}(u)^{m+1})|\text{ for all }a>0.
		\end{align*}
		We have thus shown that $u$ is a solution of the torsion  problem related to the nonlinear transparent media  
		$$
		\begin{cases}
			\dis -\operatorname{div}\left(u^m\frac{D u}{|D u|}\right) = 1 & \text{in}\;\Omega,\\
			u=0 & \text{on}\;\partial\Omega,
		\end{cases}
		$$
		in the sense of Definition \ref{def sol}.
	\end{example}
	Our next example deals with smooth  non-constant \bk  radial solutions and illustrates the lack of interplay between the sign of the operator $-\operatorname{div}\left(u^m\frac{Du}{|Du|}\right)$ and the sign of $u$.
	\begin{example} 
		Let $B_{R}$ be the ball of radius $R$ centered at the origin. For $\theta>0$, we define 
		$$
		u(x)=\frac{1}{R^{\theta}}\left(R^{\theta}-|x|^{\theta}\right).
		$$
		We will show that, despite $u$ being positive with $\Delta u\leq0$, the function $-\operatorname{div}\left(u^m \bk \frac{\nabla u}{|\nabla u|}\right)$ changes sign in the interior of $\Omega$. Indeed, a straightforward computation yields 
		$$
		|\nabla u|=\left(\frac{\theta|x|^{\theta-1}}{R^{\theta}}\right),\qquad -\Delta u=\frac{\theta |x|^{\theta-2}}{R^{\theta}}\left(N+\theta-2\right)\qquad-\Delta_1u=\frac{N-1}{|x|}.
		$$
		Consequently, using \eqref{equ comp}
		\begin{align*}
			-\operatorname{div}\left(u^{m}\frac{\nabla u}{|\nabla u|}\right)&=-mu^{m-1}\left(\frac{\theta|x|^{\theta-1}}{R^{\theta}}\right)+u^{m}\frac{(N-1)}{|x|}\\
			&=\frac{u^{m-1}}{R^\theta|x|}(R^{\theta}u(N-1)-m\theta|x|^{\theta})\\
			&=\frac{u^{m-1}}{R^\theta|x|}\left(R^{\theta}(N-1)-|x|^{\theta}(N-1+m\theta)\right).
		\end{align*}
		Thus, we have   that the problem
		$$
		\begin{cases}
			\dis -\operatorname{div}\left(u^{m}\frac{\nabla u}{|\nabla u|}\right) = u^{m-1}f(x) & \text{in}\;B_{R},\\
			u=0 & \text{on}\;\partial B_{R},	
		\end{cases}
		$$
		with
		$$
		f(x)=\frac{R^{\theta}(N-1)-|x|^{\theta}(N-1+m\theta)}{R^\theta|x|},
		$$
		admits a smooth radial solution. In particular, when $m=\theta=R=1$, $u$ solves
		$$
		-\operatorname{div}\left(u\frac{\nabla u}{|\nabla u|}\right)=\frac{N-1}{|x|}-N:=f.
		$$
		We observe that $u$ is positive and  concave   while  $-\operatorname{div}\left(u\frac{\nabla u}{|\nabla u|}\right)\in L^{N,\infty}(B_{1})$, but changes  its sign as it is  negative near the boundary of $B_{1}$. Finally, one can remark that 
		$$
		\widetilde{S}_1\|f\|_{L^{N,\infty}}=\widetilde{S}_1\omega_{N}^{\frac{1}{N}}(N-1) =1,
		$$
		which enhances  the sharpness  of the result in Corollary \ref{lemma ae conv} as $\max_{B_1} u (x) =1$. \bk 
		\bk 
	\end{example}
	\textbf{Data availability:}
	No data was used for the research described in the article.\\
	
	\textbf{Conflict of interest declaration:}
	The authors declare no conflict of interest.\\
	
	\textbf{Acknowledgments:}  F. Balducci, F. Oliva and  F. Petitta are partially supported by the Gruppo Nazionale per l’Analisi Matematica, la Probabilità e le loro Applicazioni (GNAMPA) of the Istituto Nazionale di Alta Matematica (INdAM). M.F. Stapenhorst is partially  supported by FAPESP 2022/15727-2 and 2021/12773-0.


\begin{thebibliography}{}
		\bibitem{acc} F. Alter, V. Caselles and A. Chambolle,  A characterization of convex calibrable sets in $\re^N$,  Math. Ann. 332,  329-366,  (2005)
		\bibitem{AL} A. Alvino, Sulla disuguaglianza di Sobolev in spazi di Lorentz, Boll. Un. Mat. Ital. 14, 3-11, (1977)
		
		\bibitem{ACM} L. Ambrosio, G. Crippa, S. Maniglia, Traces and fine properties of a $ BD $ class of vector fields and applications, Ann. Fac. Sci. Toulouse Math. 14 (4), 527-561, (2005)
		
		\bibitem{AFP} L. Ambrosio, N. Fusco, D. Pallara, Functions of bounded variation and free discontinuity problems, Courier Corporation, (2000)
		
		\bibitem{acm2001} F. Andreu, C. Ballester, V. Caselles, J.M. Maz\'on, The {D}irechlet problem for the total variation flow, J. Funct. Anal. 180 (2), 347-403, (2001)
		
		
		
		\bibitem{AndreuMazonMollCaselles2004} F. Andreu, J.M. Maz\'on, S. Moll, V. Caselles, The minimizing total variation flow with measure initial conditions, Comm. Contemp. Math. 6 (3), 431-494, (2004)
		
		\bibitem{acmNA} F. Andreu, V. Caselles, J.M. Mazón, A strongly degenerate quasilinear elliptic equation, Nonlinear Anal. 61 637-669, (2005) 

\bibitem{acmJEMS} F. Andreu, V. Caselles, J.M. Mazón, The Cauchy problem for a strongly degenerate quasilinear equation, J. Eur. Math. Soc.
(JEMS) 7 361-393, (2005) 
		
		\bibitem{acmMA} F. Andreu, V. Caselles, J. M. Mazón and S. Moll, The Dirichlet problem associated to the relativistic heat equation, Math. Ann. 347(1) 135-199,  (2010) 
		
		\bibitem{ACMMevo} F. Andreu, V. Caselles, J. M. Maz\'on, J.S. Moll, A diffusion equation in transparent media, J. Evolution Equations 7(1), 113-143, (2007)
		
		\bibitem{AndreuCasellesMazon2004} F. Andreu, V. Caselles, J.M. Maz\'on, Parabolic quasilinear equations minimizing linear growth functionals, Progress in mathematics 223, Birkh\"auser Verlag, Basel, (2004)
		
		\bibitem{ads} F. Andreu, A. Dall’Aglio, S. Segura de Le\'on, Bounded solutions to the 1-Laplacian equation with a critical gradient term, Asymptot. Anal. 80 (1-2), 21-43, (2012)
		
		\bibitem{A} G. Anzellotti, Pairings between measures and bounded functions and compensated compactness, Ann. Mat. Pura Appl. 135, 293-318, (1983)
		
		\bibitem{AV} A. Audrito, J. L. V\'azquez, The Fisher-KPP problem with doubly nonlinear diffusion, J. Diff. Eqns. 263 (11), 7647-7708, (2017)
		
		\bibitem{BOP} F. Balducci, F. Oliva, F. Petitta, Finite energy solutions for nonlinear elliptic equations with competing gradient, singular and $L^1$ terms, J. Diff Eqns. 391, 334-369, (2024)
		
		\bibitem{BCRS} M. Bertalmio, V. Caselles, B. Roug\'e, A. Sol\'e, TV based image restoration with local constraints, Special issue in honor of the sixtieth birthday of Stanley Osher, J. Sci. Comput. 19 (1-3), 95-122, (2003)
		
		\bibitem{B} F. E. Browder, Pseudo-monotone operators and nonlinear elliptic boundary value problems on unbounded domains, Proc. Nat. Acad. Sci. USA 74 (7), 2659-2661, (1977)
		
		
		\bibitem{calvo} J. Calvo, J. Campos, V. Caselles, O. Sánchez and J. Soler, 
Flux-saturated porous media equations and applications. EMS Surv. Math. Sci 2 131-218, (2015)
		
		\bibitem{CF} G.-Q. Chen, H. Frid, Divergence-measure fields and hyperbolic conservation laws, Arch. Ration. Mech. Anal. 147, 89-118, (1999)
		
		\bibitem{CT} M. Cicalese, C. Trombetti, Asymptotic behaviour of solutions to $p$‐Laplacian equation, Asymp. Anal. 35 (1), 27-40, (2003)
		
		\bibitem{CDC} G. Crasta, V. De Cicco, Anzellotti's pairing theory and the Gauss–Green theorem, Adv. Math. 343, 935-970, (2019)
		
		\bibitem{dass} A. Dall’Aglio, S. Segura de Le\'on, Bounded solutions to the $1$-Laplacian equation with a total variation term, Ric. Mat. 68 (2), 597-614, (2019)
		
		\bibitem{DGOP} V. De Cicco, D. Giachetti, F. Oliva, F. Petitta, The Dirichlet problem for singular elliptic equations with general nonlinearities, Calculus of Variations and Partial Differential Equations, 58, 1-40, (2019)
		
		\bibitem{Demengel2005} F. Demengel, Some existence results for noncoercive "1-Laplacian" operator, Asymptot. Anal. 43 (4), 287-322, (2005)
		
		\bibitem{GOP} D. Giachetti, F. Oliva, F. Petitta, 1-Laplacian type problems with strongly singular nonlinearities and gradient terms, Commun. Contemp. Math. 24 (10), Paper No. 2150081, 40 pp, (2022)
		
		\bibitem{GMP} L. Giacomelli, S. Moll, F. Petitta, Nonlinear diffusion in transparent media: the resolvent equation, Adv. Calc. Var. 11, 405-432, (2018)
		
		\bibitem{K} B. Kawohl, On a family of torsional creep problems, J. Reine Angew. Math. 410, 1-22, (1990)
		
		\bibitem{ka} B. Kawohl, From $p$-Laplace to mean curvature operator and related questions, In: Progress in Partial Differential Equations: the Metz Surveys, Pitman Res. Notes Math. Ser. 249, 40-56, (1991)
		
		\bibitem{lops} M. Latorre, F. Oliva, F. Petitta, S. Segura de Le\'{o}n, The Dirichlet problem for the 1-Laplacian with a general singular term and $L^1$-data, Nonlinearity 34 (3), 1791-1816, (2021)
		
		\bibitem{LTS} M. Latorre, S. Segura de León, Elliptic equations involving the $1-$Laplacian and a total variation term with $L^{N,\infty}$-data, Atti Accad. Naz. Lincei Rend. Lincei Mat. Appl. 28 (4), 817-859, (2017)
		
		\bibitem{ON} R. O'Neil, Integral transforms and tensor products on Orlicz spaces and $L^{p,q}$ spaces, J. Anal. Math. 21, 1-276, (1968)
		
		\bibitem{MST2009} A. Mercaldo, S. Segura de Le\'{o}n, C. Trombetti, On the solutions to 1-{L}aplacian equation with $L^1$ data, J. Funct. Anal. 256 (8), 2387-2416, (2009)
		
		\bibitem{MST1} A. Mercaldo, S. Segura de Le\'on, C. Trombetti, On the behaviour of the solutions to $p-$Laplacian equations as $p$ goes to $1$, Publ. Mat. 52, 377-411, (2008)
		
		\bibitem{M} Y. Meyer, Oscillating Patterns in Image Processing and Nonlinear Evolution Equations: The Fifteenth Dean Jacqueline B. Lewis Memorial Lectures, American Mathematical Society, Providence, RI, (2001)
		
		\bibitem{O} F. Oliva, Regularizing effect of absorption terms in singular problems, J. Math. Anal. Appl. 472 (1), 1136-1161, (2019)
		
		\bibitem{OPS} F. Oliva, F. Petitta, S. Segura de Le\'on, The role of absorption terms in Dirichlet problems for the prescribed mean curvature equation, Nonlin. Diff. Eqns. and Appl. 31 (4), 53, (2024)
		
		\bibitem{OsSe} S. Osher, J. Sethian, Fronts propagating with curvature-dependent speed: algorithms based on Hamilton-Jacobi formulations, Journal of Computational Physics, 79 (1), 12-49, (1998)
		
		\bibitem{ros} P. Rosenau, Tempered diffusion: A transport process with propagating front and inertial delay, Phys. Rev. A 146, 7371-7374, (1992)
		
		\bibitem{ROF} L.I. Rudin, S. Osher, E. Fatemi, Nonlinear total variation based noise removal algorithms, Physica D 60, 259-268, (1992)
		
		\bibitem{Sapiro} G. Sapiro, Geometric partial differential equations and image analysis, Cambridge University Press, (2001)
		
		\bibitem{S} G. Stampacchia, Le probl\`eme de Dirichlet pour les \'equations elliptiques du second ordre \`a coefficients discontinus, Ann. Inst. Fourier Grenoble 15, 189-258, (1965)
		
		\bibitem{Z} W. P. Ziemer, Weakly Differentiable Functions, Springer, (1989)
		
	\end{thebibliography}
\end{document}